\documentclass[10pt,reqno]{amsart}

\usepackage{amssymb,epsfig,graphics}
\usepackage{amsmath,enumerate}

\usepackage{geometry}        
\usepackage[parfill]{parskip}  
\usepackage{graphicx}
\usepackage{textcomp} 

\usepackage[applemac]{inputenc}
\usepackage{ae, aeguill}

\usepackage{pdfsync}

\newcommand{\cleqn}{\setcounter{equation}{0}} 
\cleqn 

\newtheorem{thm}{Theorem}[section]
\newtheorem{prop}[thm]{Proposition}
\newtheorem{defn}[thm]{Definition}
\newtheorem{cor}[thm]{Corollary}
\newtheorem{lem}[thm]{Lemma}
\newtheorem{sub-lemma}[thm]{Sub-Lemma}
\newtheorem{rem}[thm]{Remark}

\let\io =\displaystyle
\let\± =\displaystyle

\def \R{{\mathbb R}}
\def \Z{{\mathbb Z}}

\def\C{{\mathcal{ C}}}

\def\mE{{\mathbb{E}}}
\def\P{{\mathbb{ P}}}

\def\I{{\cal I}}
\def\L{{\mathcal{L} }}
\def\I{{\mathcal I}}

\def\cS{{\mathcal{S} }}

\def\o{\omega}

\def\1{{\mathbf 1}}

\def\capa{{\mathop{\rm Cap}}}

\def\er{{\stackrel {\rightharpoonup} { e} }}
\def\el{{\stackrel {\leftharpoonup} { e} }}

\def\eps{\varepsilon}
\def\es{{{\rm Es}}}

\def\bi{\bigbreak}
\def\bino{\bigbreak\noindent}
\def\sm{\smallbreak}

\begin{document}

\title{\textbf{ Existence of the harmonic measure for random walks
on  graphs and in random environments }}
\author{Daniel Boivin and  Clément Rau}
\date{}

\begin{abstract}
We give a sufficient condition for the existence of the
harmonic measure from infinity of transient random walks on weighted graphs.
In particular, this condition is verified by the random conductance
model on $\Z^d$, $d\geq 3$, when the conductances
 are i.i.d. and the bonds with positive conductance percolate.
The harmonic measure from infinity also exists
for  random walks on  supercritical clusters of ${\mathbb Z}^2$.
This is proved using  results of Barlow (2004).
\bigbreak\noindent
\bigbreak\noindent
\end{abstract}

\maketitle
\markboth{DANIEL BOIVIN and CLEMENT RAU}{EXISTENCE OF THE HARMONIC MEASURE}
\thanks{}


\section{Introduction and results}
\cleqn

The harmonic measure from infinity  of a closed subset $A$ of $\R^d$, $d\geq 2$,
is the hitting distribution of the set $A $ by a $d$-dimensional
Brownian motion started at infinity. A detailed description of this measure is given
by Mörters and Peres in \cite[section 3.4]{MoPe10}.

Similarly, given a Markov chain on an infinite graph,
the harmonic measure of a finite subset of the graph
is defined as the hitting distribution of the set by the Markov chain starting at infinity.
 The  existence of the harmonic measure for the simple
symmetric random walk on $\Z^d$  is shown by Lawler
in \cite [chapter 2]{Law96} and it is extended to a wider class of random
walks on $\Z^d$  by Lawler and Limic  in
 \cite [section 6.5]{LaLi10}.

From these results, one might expect that the existence of the harmonic
measure for a Markov chain on $\Z^d$, $d\geq 2$,
relies on its Green function asymptotics.
The goal of this paper is to show that
actually, the existence of the harmonic measure is a
 fairly robust result in the sense that  it exists for a random
walk on a weighted graph
as soon as there is a weak form of Harnack inequality.
In particular, it is verified by a large family of fractal-like graphs
and by random conductance models
on $\Z^d$, $d\geq3$, given by a sequence of i.i.d. conductances
as soon as there is percolation of the positive conductances.
This is done using recent estimates of Andres, Barlow, Deuschel and
Hambly \cite{ABDH10}.

In the recurrent case, although we do not give a general sufficient
condition, we  show the existence
of the harmonic measure for the random walk on the supercritical
cluster of $\Z^2$
using   estimates of Barlow \cite{Bar04} and Barlow and Hambly
\cite{BaHa09}.

The results of \cite{ABDH10} for the random conductance model
are part of a long series of works which go back
to  homogenization
of divergence form elliptic operators with random coefficients
and to the investigation of the properties of the supercritical percolation
cluster.

Some highlights of the properties of the random walk on the
supercritical percolation cluster of $\Z^d$ is the 
proof
of the Liouville property for bounded harmonic
functions (see Kaimanovich \cite{Kai90} and \cite{BLS99})
 and the proof of  the transience  of the
walk when $d\geq 3$ by Grimmett, Kesten and Zhang \cite{GKZ93}.

In \cite{Bar04}, Barlow proved upper and lower gaussian estimates for
the probability transitions of a random walk on the supercritical
percolation cluster. These are then used to prove a Harnack inequality
\cite [Theorem 3]{Bar04}. The Liouville
property for positive harmonic functions on the percolation cluster
follows as well as an  estimate of the mean-square displacement of the walk.

Barlow's upper gaussian estimates were also used to prove 
the invariance principle for the random walk on supercritical percolation clusters
by \cite{SiSz04}, \cite{MaPi07}, \cite{BeBi07}. 
The invariance principle for the random walk on $\Z^d$ 
with independent conductances that are bounded below is proved in
\cite{BaDe10}. 

Here we show that the existence of the harmonic
measure  follows from the Green function
estimates of \cite[Theorem 1.2]{ABDH10}.
In the case of the two-dimensional percolation cluster,
we use both the elliptic and the parabolic Harnack inequalities
 of \cite{Bar04} and \cite{BaHa09}.

 Whenever the harmonic measure from  infinity exists, one can study  
external  diffusion-limited aggregates. Their growth is determined
by the harmonic measure which can also be interpreted as
the distribution of an electric field on the surface
of a grounded conductor with fixed charge of unity.
Recent simulations by  physicists of the harmonic measure in $\Z^d$ 
can be found in \cite{ASSZ09} and of percolation and Ising clusters
in \cite{ASZ08}. Analytic predictions  for  the harmonic measure of
two dimensional clusters are given by Duplantier in
\cite{Dup99} and \cite{Dup00}. See also the survey paper \cite{Bar93}.

In contrast, for the  internal diffusion-limited aggregates
of random walks on percolation clusters,
the limiting shape is described in \cite{She10} and \cite{DLYY}.

The values of the constants $c, C, C' \ldots$ may change at each
appearance
but they are always strictly positive and they do not depend 
on the environment. The minimum of $a$ and $b$ and the maximum of $a$ and $b$ are
respectively denoted by $a\wedge b$ and by $a\vee b$.

\subsection{Reversible random walks}

A weighted graph $(\Gamma, a)$ is given by a countably infinite set
$\Gamma$ and a symmetric function
$$a : \Gamma\times\Gamma\to\lbrack 0 ; \infty\lbrack$$
which verifies $a(x,y) = a(y,x)$ for all $x,y\in\Gamma$ and
$$ \pi(x) :=  \sum_{y\in\Gamma} a(x,y)>0\  \hbox{ for all }\  x\in\Gamma.$$

The weight $a(x,y)$ is also called the conductance
of the edge connecting $x$ and $y$ since the weighted graph can be
interpreted as an electrical or thermic network.

Given a weighted graph $(\Gamma, a)$, we will write
$x\sim y $ if $a(x,y)>0$. We will always assume that
 $(\Gamma, \sim )$ is an infinite, locally finite countable graph
without multiple edges. 
A path of length $n$ from $x$ to $y$ is a sequence
 $x_0, x_1,\ldots, x_n$ in $\Gamma$
such that $x_0=x$, $x_n=y$ and $x_{i-1}\sim x_{i}$
for all $1\leq i\leq n$.
The weighted graph $(\Gamma, a)$ is said to be
connected if $(\Gamma, \sim )$
is a connected graph, that is, for all $x,y\in\Gamma$
 there is a path of finite length from $x$ to $y$.
The graph distance between two vertices $x,y\in\Gamma$
will be denoted by  $D(x,y)$. It is the minimal length of a path from
$x$ to $y$ in the graph $(\Gamma,\sim)$. 
 The ball centered at $x\in\Gamma$ of radius $R$ will be denoted by
$B(x,R) := \{ y\in \Gamma; \ D(x,y) < R \}$. 

The random walk on the weighted graph $(\Gamma, a)$
is the Markov chain on $\Gamma$
with transition probabilities given by
\begin{equation}
p(x,y):=\frac{a(x,y) }{\pi(x)}, \quad x,y\in \Gamma.
\label{defprobtrans}
\end{equation}

We denote by $P_x$ the law of the random walk
starting at the vertex $x\in \Gamma$.
The corresponding expectation is denoted by  $E_x $.
The random walk   admits  reversible measures which are proportional
to  the measure $\pi(\cdot)$.

For $A\subset \Gamma$, we have the following definitions

\hskip1cm $\partial A  :=\{ y\in\Gamma ; \  y\notin A \
\hbox{and there is}\  x\in A \  \hbox{with} \  x\sim y\}$
and $\io \overline A :=\partial A\cup A$,

 \hskip1cm $\io \tau_A :=\inf \{k\geq 1; \ X_k\in A\}$ and
 $\overline{\tau}_A :=\inf \{k\geq 0; \ X_k\in A\}$

\hskip1cm  with the convention that $\inf\emptyset = \infty$,

\hskip1cm $D(x, A):=\inf\{D(x,y) ; y\in A\}$,

\hskip1cm  and for $u : \overline A\to\R$ the Laplacian is defined by
$$\L u(x) := \sum_{y\sim x} p(x,y)(u(y) - u(x)),\quad x\in A.$$

\hskip1cm A function $u : \overline A \to \R$ is {\it harmonic}  in $A$ if
for all $x\in A$, $ (\L u)(x) = 0$.

The {\it Green function} of the  random walk is defined by
\begin{equation}\label{defgreen}
G(x,y) := \sum_{k=0}^\infty p(x,y,k),\quad x,y\in\Gamma
\end{equation}
where $ p(x,y,k) :=  P_x(X_k=y)$ are the transition probabilities
of the walk. Note that $G(\cdot , y)$ is harmonic in $\Gamma\setminus \{y\}$.

For irreducible Markov chains, if $G(x,y)<\infty$ for some
$x,y\in\Gamma$ then $G(x,y)<\infty$ for all $x,y\in\Gamma$.
The random walk is {\it recurrent} if $G(x,y)=\infty$ for some
$x,y\in\Gamma$
otherwise we say that the walk is  {\it transient}.

\subsection{Results on the existence of the harmonic measure}

Let $X=(X_j)$ be a random walk on a connected weighted graph $(\Gamma,a)$.

The {\it hitting distribution } of a set $A$ starting from $x \in
\Gamma$ is given by 
$$ H_A(x,y):=P_x (X_{\tau_A} =y),\quad y\in A. $$
If $P_x ( \tau_A <+\infty) >0$ , we also consider
$$ \overline{H}_A(x,y):=P_x (X_{\tau_A} =y | \tau_A <+\infty). $$

The {\it harmonic measure} on a finite subset 
$A$ of $\Gamma $ is the hitting distribution from infinity, if it exists,
\begin{equation}\label{harmlimit}
{\bf H}_A(y) := \lim_{D(x, A)\to\infty} {\overline{ H}}_A(x,y),\quad y\in A.
\end{equation}

Our goal is to prove the existence of the harmonic measure
for all finite subsets of various weighted graphs.
The proof of the existence of the harmonic
measure given in \cite [section 6.5]{LaLi10} for random walks on $\Z^d$, 
relies on a Harnack
inequality and on Green function estimates.  Actually, it turns out
that only a weak form of Harnack inequality is needed.

In Theorem I, we  show that a weaker version of Harnack
inequality 
is a sufficient condition for the existence of the harmonic measure of
 transient graph. 
Moreover,  weak estimates of the Green function
imply the weak Harnack inequality.

As it happens for Brownian motion and for  simple random walks (see
for instance \cite{MoPe10},
\cite{Law96}),
the harmonic measure can be  expressed
in terms of  capacities.

The {\it capacity } of  $A$ with respect to $B$, for $A\subset
B\subset \Gamma$,  is defined by 
$$ {\mathop{\rm Cap}}_B(A):=\sum_{x\in A}\pi(x)P_x(\overline\tau_{B^c}<\tau_A). $$

The {\it escape probability} of a set $A$ is
defined by
$\io \es_A(x) :=  P_x(\tau_A=\infty)$ and the capacity of
a finite subset $A\subset\Gamma$ is defined by
$$ {\mathop{\rm Cap}}(A):=\sum_{x\in A}\pi(x)\es_A(x). $$

The main result for transient graphs is the existence of the harmonic measure for random
walks which verify the following weak Harnack inequality.

\begin{defn}\label{wHarineq}
We say that a weighted graph $(\Gamma, a)$
satisfies {\bf wH(}$C${\bf )},
 the weak Harnack inequality, if there is a constant $C \geq 1$ such that
  for all $x\in \Gamma$ and for all $R>0$ there is $R'=R'(x,R)$
such that  for any positive harmonic
function $u$ on $B(x, R')$,
$$\max_{B(x,R)} u \leq C \min_{B(x,R)} u.$$
\end{defn}

\bf Theorem I.\it\ \ 
Let $(\Gamma, a)$ be a weighted graph.

If $(\Gamma, a)$ is connected, transient and if it verifies
the weak Harnack inequality  {\bf wH(}$C${\bf )},

then for any finite subset $A \subset \Gamma$  the harmonic measure on
$A$ exists.
That is, for all $y\in A$,  the limit (\ref{harmlimit}) exists.

Moreover, we have:
$$ \lim_{D(x, A)\to\infty}  \overline{H}_A(x,y)= \lim_{m \rightarrow + \infty } H_A^m(y),$$
where, for $m$ large enough, 
$$ H_A^m(y)=  \frac{ \pi(y) P_y ( \tau_A > \tau_{\partial B(x_0,m)
  }) }{ {\mathop{\rm Cap}}_m(A)  }  \  $$
 where  $\io {\mathop{\rm
    Cap}}_m(A)$
is the capacity of $A$ with respect to $ B (x_0,m)$  for some
$x_0\in\Gamma$.
The limit does not depend on the choice of $x_0$.
\rm\bigskip

The following Green function estimates
imply the weak Harnack inequality.

 \begin{defn} 
We say that a  weighted graph $(\Gamma, a)$
satisfies the Green function estimate  {\bf\ref{GEgamma}}  for $\gamma>0$
if there are constants $0< C_i \leq C_s <\infty$
 and  if for all $z \in \Gamma$, there exists $R_z<\infty$
such that for all $x,y \in \Gamma $ with $D(x,y)\geq R_x \wedge R_y$ we have:
 \begin{equation}
 \tag{${\bf{GE_{\gamma}}}$}  \label{GEgamma}
 \frac{C_i}{D(x,y)^{\gamma}} 
\leq  G(x,y)\leq \frac{C_s}{D(x,y)^{\gamma}}.  
 \end{equation}
\end{defn}
This condition is a weak version of \cite [Definition 1]{Tel00}
where $\gamma$ is called a Greenian index.
It is used by Telcs \cite {Tel00} to give an upper bound for the probability
transitions of a Markov chain in terms of the growth rate of the volume and of the
Greenian index.

\begin{prop} \label{tintin}
Let $(\Gamma, a)$ be a weighted graph which verifies
{\bf (\ref{GEgamma}) }  for some  $\gamma > 0$.
Then the graph is connected, transient and 
  {\bf wH(}$C${\bf )} holds with $\io C=2^{\gamma} \frac{C_s}{C_i}. $
\end{prop}

In the following corollaries, we describe some weighted graphs
 where the harmonic measure from infinity exists.

A weighted graph $(\Gamma, a)$ is said to be \it uniformly elliptic \rm
if there is a constant $c\geq 1$ such that for all edges $e$,
\begin{equation} \label{EL}
c^{-1}\leq a(e)\leq c .
\end{equation}

\begin{cor} \label{HDunif}
Let $(\Z^d,a)$, $d\geq 3$, be a  uniformly elliptic graph.

Then  for all finite subsets $A$ of $\Z^d$ and
for all $y\in A$,  the limit (\ref{harmlimit}) exists.

Moreover, we have:
$$ \lim_{{ |x|  \rightarrow + \infty}}  \overline{H}_A(x,y)= \lim_{m \rightarrow + \infty } H_A^m(y),$$
where $\io H_A^m(y)=  
\frac{ \pi(y) P_y ( \tau_A > \tau_{\partial B(0,m) }) }
{ {\mathop{\rm Cap}}_m(A)  }.$
\end{cor}

Indeed, by \cite [Proposition 4.2]{Del99}
the  Green function of a uniformly elliptic graph $(\Z^d,a)$,
$d\geq 3$, verifies the estimates {\bf (\ref{GEgamma}) }  with $\gamma =
d-2$.
The existence of the harmonic measure then follows from proposition
\ref{tintin} and Theorem I.

The harmonic measure also exists for a large class of
fractal like graphs with some regularity properties.
Various examples are given in \cite{Bar04b}.
See also \cite[section 1.1]{Tel06} and the references therein.

The volume of a ball $B(x,R)$ is defined by $V(x,R) := \sum_{x\in
  B(x,R)} \pi(x)$
and the mean exit time from the ball is $E(x,R):= E_x(\sigma_R)$
where $\sigma_R := \inf\{k\geq 0 ; X_k\notin B(x,R)\}$.

A weighted graph $(\Gamma, a)$  has \it polynomial volume growth with exponent
$\alpha>0$  \rm if there is a constant $c\geq 1$ such that for all $x\in\Gamma$
and for all $R\geq 1$, 
\begin{equation} \label{Va}\tag{{\bf $V_\alpha$}}  
 c^{-1} R^\alpha\leq  V(x,R)\leq c R^\alpha.
\end{equation}

A weighted graph $(\Gamma, a)$  has \it polynomial mean exit time with exponent
$\beta>0$  \rm  if there is a constant $c\geq 1$ such that for all $x\in\Gamma$
and for all $R\geq 1$,
\begin{equation} \label{Eb}\tag{{\bf $E_\beta$}}  
c^{-1} R^\beta\leq  E(x,R)\leq c R^\beta.
\end{equation}

As noticed in \cite[Theorem 3.1]{Bar04b}, 
by  \cite[Theorem 5.7 and Theorem 6.1]{GrTe02},
if a weighted graph verifies  (\ref{Va}) and  (\ref{Eb}) for
$\alpha > \beta \geq 2$ and the elliptic Harnack inequality  {\bf H(}$C${\bf )}
then it is transient. Hence we obtain the following corollary of
theorem I.

\begin{cor}
Let $(\Gamma, a)$ be a weighted graph verifying (\ref{Va}) and  (\ref{Eb}) for
$\alpha > \beta \geq 2$ and the elliptic Harnack inequality  {\bf H(}$C${\bf )}.
Then for all finite subsets $A\subset\Gamma$ and $y\in A$
the limit (\ref{harmlimit}) exists.
\end{cor}

The harmonic measure from infinity also exists for
random walks in random environment and in particular for
the random walk on the supercritical percolation cluster.
 Before stating this result, we give a brief
description of the percolation model. See \cite{Kes82} for more details.

Consider the  lattice $\Z^d$, $d\geq 2$, where $x\sim y$ if $\vert
x-y\vert_1=1$ where $\vert\cdot\vert_1$ is the $\ell_1$-distance.
Denote the set of edges by $\mE^d$.

 Assume that  $(a(e) ; e\in \mE^d)$ are i.i.d. non-negative
random variables on a probability space $(\Omega, \P)$.
Call a bond $e$ open if $a(e) >0 $ and closed if $a(e) =0$.
Let $p=\P(a(e)>0)$. 
By percolation theory,  there exists a critical value 
$p_c=p_c(\Z^d)\in ]0;1[$ such that for $p<p_c$, $\P$ almost surely, 
 all open clusters of $\omega$ are finite and for $p>p_c,$ $\P$ almost
 surely,  there is a unique infinite  cluster of open edges which is
 called the supercritical cluster. It will be denoted
 by $\C_\infty = \C_\infty(\omega)$. The edges of this graph are the open edges of the
 cluster and the endpoints of these edges are the vertices of the
 graph. 

For $x,y\in \C_\infty(\omega)$, we will write 
 $x\sim y$ if the edge with endpoints $x$ and $y$ is open.
The transition probabilities of the random walk on $\C_\infty(\omega)$ are given by 
 (\ref{defprobtrans}). The law of the paths starting at $x\in
 \C_\infty(\omega)$ will be denoted by $P_x^{\omega}$.
The random walk on the supercritical
 percolation cluster corresponds to the case of Bernoulli random
 variables. In this case, we will write $\P_p$ instead of $\P$.

$D_\omega(x,y)$ will denote the graph distance between $x$ and $y$ in the graph $\C_\infty(\o)$ and the
  ball centered at $x\in\C_\infty(\o)$ of radius $R$ will be denoted by
$B_\omega(x,R) = \{ y\in \C_\infty(\o); \ D_\omega(x,y) < R \}$.\\

The existence of the harmonic measure for $\Z^d$, $d\geq 3$, with
i.i.d. conductances,  is given in corollary \ref{HDperco} below. 
It follows from  the Green function estimates of \cite[Theorem 1.2a]{ABDH10}.
A weaker condition which might hold even if the conductances are not
i.i.d. is given in  \cite[Theorem 6.1]{BaDe10}.

\begin{cor}\label{HDperco}
Let $(\Z^d, a)$, $d\geq 3$, be a weighted graph where the weights
  $(a(e) ; e\in \mE^d)$ are i.i.d. non-negative 
random variables on a probability space $(\Omega, \P)$
which verify 
  $$\P(a(e)>0) > p_c(\Z^d).$$

Then there exist constants $C_i$,  $C_s$, which depend on $\P$ and $d$,
 and $\Omega_1 \subset \Omega$ with $\P(\Omega_1)=1$ such that 
for each $\omega \in \Omega_1$, {\bf (\ref{GEgamma}) }  holds in $ \C_\infty(\omega)$
with the constants $C_i$ and $C_s$ and with $\gamma = d-2$.

For any finite subset $A$   of   $\C_\infty$ and for all $y\in A$,
the limit (\ref{harmlimit}) exists.

Moreover, we have:
$$ \lim_{{ |x|  \rightarrow + \infty}, {x\in \C_\infty}}
\overline{H}_A(x,y) = \lim_{m \rightarrow + \infty } H_A^m(y),$$
where $H_A^m(y)=  \frac{ \pi(y) P_y^{\omega} ( \tau_A > \tau_{\partial
    B_{\omega} (x_0,m) }) }
{     {\mathop{\rm Cap}}_m(A) }  \  $ for some $x_0\in \C_\infty$ and for $ m$ large enough.
\end{cor}

In \cite{ABDH10}, both the  constant speed random walk and
the variable speed random walk are considered. 
From the expression of their generators one immediately sees
that they have the same harmonic functions as the discrete time random
walk considered here.
 Moreover, since they are a time change of each other, the Green
 function
is the same.
Hence, by \cite [Theorem 1.2 a]{ABDH10}
the  Green function of a uniformly elliptic graph $(\Z^d,a)$,
$d\geq 3$, verifies the estimates {\bf (\ref{GEgamma}) }  with $\gamma =
d-2$.
The existence of the harmonic measure then follows from proposition
\ref{tintin} and Theorem I.

The harmonic mesure from infinity also exists for recurrent graphs.
The main result here is the existence of the harmonic measure
for all finite subsets of two-dimensional supercritical percolation clusters.

\bf Theorem II.\it\ \ 
Let $(\Z^2, a)$  be a weighted graph where the weights
  $(a(e) ; e\in \mE^2)$ are i.i.d.
random variables on a probability space $(\Omega, {\P}_p)$
which verify 
  $$p={\P}_p(a(e)=1)=1-{\P}_p(a(e)=0) > p_c(\Z^2).$$
Then ${\P}_p$ almost surely,
for any finite subset $A$   of   $\C_\infty(\o)$ and for all $y\in A$,
the limit (\ref{harmlimit}) exists.

An expression for the value of the limit  (\ref{harmlimit}) is given
in equation  (\ref{ident}).
\rm\bigskip

\begin{thm}\label{HDunif2}
If  $(\Z^2,a)$  is a uniformly elliptic weighted graph
then for all finite subsets $A\subset {\mathbb Z}^2$
and  for all $y\in A$,  the limit (\ref{harmlimit}) exists.
\end{thm}

\begin{rem}
Note that on a regular tree, the harmonic mesure from infinity does
not exist for any set $A$ which contains at least two vertices.
It would be interesting to investigate the links between the
Poisson boundary 
 of a graph and the existence of the harmonic measures.
In particular, the triviality of the Poisson boundary
does not imply the existence of the harmonic measure
as is shown by the lamplighter group $\Z^2\wr \Z/2\Z$.
See \cite{Sav10} and the references therein.
\end{rem}

Various forms of Harnack inequality that will be used both for
transient
graphs and for recurrent graphs are gathered in section \ref{HarIneq}.
The proof of theorem I
is given in section \ref{transient} while Theorem II
is proved in section \ref{recurrent}. 
The last section contains the proof of the annulus Harnack inequality
that is used in the proof of Theorem II.

\section{Harnack inequalities}\label{HarIneq}
\cleqn

We start by recalling a classical form of the Harnack inequality on a graph.
Then we give related inequalities and weaker versions.

\begin{defn}\label{Harineq}
We say that a weighted graph $(\Gamma, a)$ satisfies {\bf H(}$C${\bf )}, 
  the Harnack inequality with shrinking parameter  $M> 1$, 
if  there is a constant $C <\infty $ such that
for all $x\in \Gamma$ and $R>0$, and
for any positive harmonic function $u$ on $B(x, MR)$,
$$\max_{B(x,R)} u \leq C \min_{B(x,R)} u.$$
\end{defn}

In our context, we will use the weak form of Harnack inequality
given in definition \ref{wHarineq}.  We rewrite this definition
under a form similar to definition \ref{Harineq}. The proofs
will be given with these notations.

\begin{defn}
We say that a weighted graph $(\Gamma, a)$
satisfies {\bf wH(}$C${\bf )},
 the weak Harnack inequality, if there is a constant $C>0$ such that
  for all $x\in \Gamma$ and for all $R>0$ there is $M_{x,R}\geq 2$
such that for all $M>M_{x,R} $ and for any positive harmonic
function $u$ on $B(x, MR)$,
$$\max_{B(x,R)} u \leq C \min_{B(x,R)} u.$$
\end{defn}

Barlow \cite[Theorem 3]{Bar04} showed that the supercritical percolation
cluster verifies another form of Harnack inequality. 
However,  by corollary
\ref{HDperco} and proposition \ref{tintin} below,  the random walk
on the supercritical percolation
cluster also verifies  {\bf wH(}$C${\bf )}. Given below is a
Harnack inequality  under the form that will be most useful to us.
It is an immediate consequence of Theorem 5.11,
proposition 6.11 and of (0.5) of Barlow's work \cite{Bar04}.

\bf Harnack Inequality for the percolation cluster\ \rm \cite{Bar04}.\it\  
Let $d\geq 2$ and let $p>p_c(\Z^d)$.
There exists $c_1=c_1(p,d)$ and $\Omega_1 \subset\Omega$
with $\P_p(\Omega_1)=1$, and $R_0(x, \omega)$ such that 
$3\leq R_0(x,\omega)<\infty $ for each
$\omega \in \Omega_1, \ x\in \C_\infty(\omega).$

If $R\geq R_0(x,\omega)$ and if $D(x,z)\leq \frac{1}{3}R\ln R$
and if $u : \overline{B(z,R)}\to \R$ is positive and harmonic
in $B(z,R)$, then
\begin{eqnarray}
  \label{HarnackBarlow1}
  \max_{B(z,R/2)} u \leq c_1 \min_{B(z,R/2)} u.
\end{eqnarray}

Moreover, there are positive constants $c_2, c_3$ and $\eps$ which
depend on $p$ and $d$ such that the tail of $R_0(x, \omega)$ satisfies
\begin{eqnarray}
  \label{HarnackBarlow2}
 \P_p(x\in\C_{\infty}, R_0(x, \cdot ) \geq n)\leq c_2\exp(-c_3 n^{\eps}).
\end{eqnarray}
\rm


In the proof of Theorem I, we will need the H\"older
continuity of harmonic functions. It is a consequence of the weak Harnack inequality. Property {\bf wH(}$C${\bf )} leads  to the following lemma.

\begin{lem}[weak Hölder continuity]\label{holder}
Let $(\Gamma, a)$ be a weighted graph which verifies {\bf{wH}}(C)
with shrinking parameters $(M_{x,R}; x\in\Gamma, R>0)$
where $M_{x,R}\geq 2$ for all $x\in \Gamma$ and $R>0$.

  Then there exist $\nu >0, \ c>0$  such that for all $x_0\in
  \Gamma$, $R>0$, $M\geq M_{x_0,R}$ 
and for any positive harmonic function $u$  on $B(x_0,MR)$ and $x \in B(x_0,R)$, 
$$  |u(x)-u(x_0)| \leq     c \  (\frac{D(x_0,x)}{R}  )^{\nu}\max_{B(x_0,MR)} u.$$
\end{lem}

\begin{proof} Let $x_0\in \Gamma$ and $R>0$. Then for all 
$M\geq M_{x_0,R} $ and $R'\leq R$, if  $u$ is a positive harmonic
function on $B(x_0,MR)$ then
$$\max_{B(x_0,R')} u \leq \max_{B(x_0, R)} u \leq C\min_{B(x_0,R)} u
\leq C\min_{B(x_0,R')} u.
$$
Let  $$V(i):=\max_{B(x_0,2^i)} u - \min_{B(x_0,2^i)} u.$$ 

 Then for  $2^i\leq R$,  the functions
$ u-\min_{B(x_0,2^{i+1})}u $ and $\max_{B(x_0,2^{i+1})} u -u$   are
harmonic in $B(x_0,MR)$. 
Then by the  weak Harnack inequality on $B(x_0,2^i)$,
  $$ V(i)+V(i+1) \leq C [V(i+1)-V(i)].$$
And so, we deduce  that there exists $\lambda<1$ such that
$$V(i) \leq \lambda \  V(i+1).$$

For any $x\in B(x_0,R)$, we can find $N_1$ such  that
 $2^{N_1-1} \leq D(x_0,x) \leq 2^{N_1}$.  
 Then $$  |u(x)-u(x_0)| \leq V(N_1).$$ 
 
Let $N_2$ be such that 
 $2^{N_2} \leq R < 2^{N_2+1}$.
Then, since $2^{N_2+1}\leq M R$,
 $$ V(N_1)\leq \lambda^{N_2-N_1+1} \ V(N_2+1)$$
and in particular, 
$$|u(x)-u(x_0)| \leq   c  \  \left( \frac{D(x_0,x)}{R}  \right)
^{\nu}\max_{B(x_0,MR)} u$$
where  $\nu>0$ solves $\lambda^{-1} = 2^\nu$ and $c>0$ is a constant.
 \end{proof}

Similarly, from Harnack inequality for  the supercritical
cluster (\ref{HarnackBarlow1}), 
we have the following H\"older continuity property.

\begin{prop}\label{holderperco}
Let $d\geq 2$ and let $p>p_c(\Z^d)$.
Let $\Omega_1$ and $R_0(x,\omega)$ be given
by the  Harnack inequality for  the supercritical
cluster.
Then there exist positive constants $\nu$ and  $c$ such that for each
$\omega \in \Omega_1, \ x_0\in \C_\infty(\omega)$
if $R\geq R_0(x_0,\omega) $ and 
$u$
is a positive harmonic function on $B_{\omega} (x_0, R)$ then,
for all $x,y \in B_{\omega}(x_0,R/2)$, 
$$|u(x)-u(y)| \leq c \  \left(
  \frac{D(x,y)}{R}\right)^{\nu}\max_{B(x_0, R)} u.$$
\end{prop}

We will also need a Harnack inequality in the annulus
of the two-dimensional supercritical percolation cluster.
It follows from  results of Barlow
\cite{Bar04}, a percolation result due to Kesten 
\cite{Kes82} and the following estimates of Antal and Pisztora
\cite[Theorem 1.1 and Corollary 1.3]{AnPi96}. 

For $d\geq 2$ and $p>p_c(\Z^d)$,
there is a constant $\mu=\mu(p, d)\geq 1$  
such that 
\begin{equation} \label{AP1}
  \limsup_{\vert x\vert_1\to\infty}\frac{1}{\vert x\vert_1}\ln
  \P_p[x_0,x\in\C_\infty, D(x_0,x)>\mu \vert x\vert_1] < 0
\end{equation}
and, $\P_p$ almost surely,
  for $x_0\in\C_\infty$ and for all $x\in\C_\infty$
such that $D(x_0,x)$ is sufficiently large
\begin{equation} \label{AP2}
  D(x_0,x)\leq \mu \vert x-x_0\vert_1.
\end{equation}

\begin{prop}\label{annulusharnack1}
Let $p>p_c(\Z^2)$.
There is a constant $C>0$ such that ${\mathbb P}_p$-a.s.,
for all $x_0\in \C_{\infty} $ and $r>0$,  if $m$ is large enough,\\
 then for any positive function $u$ harmonic in
$\± B(x_0, 3\mu m)\setminus B(x_0, r )  $,
$$\max_{x ; D (x_0, x) =m} u(x)\leq C \min_{x ; D (x_0, x) =m} u(x)
$$
where $\mu$ is the constant   that appears in (\ref{AP2}).
\end{prop}

Since we
need a construction that is done in section \ref{Kestensgrid},
the proof of this Harnack inequality   is postponed
to section \ref{proofAH1}.

\section{Proofs for transient graphs}\label{transient}
\cleqn

In this section, we prove proposition \ref{tintin} and Theorem I.

\begin{proof} [Proof of proposition \ref{tintin}.]

The key ingredient to prove  proposition \ref{tintin}
is given by Boukricha's lemma \cite{Bou79}. 
See also \cite[p. 37]{Tel06}.
Roughly speaking, this lemma ensures that a Harnack inequality holds for
general positive  harmonic functions as soon as a Harnack inequality
holds for the Green function in an annulus.

For $x\in\Gamma$ and   $R>0$, let 
\begin{equation}
M_{x,R} = 3 \vee \frac{1}{R}\max_{w\in B(x,R)}  R_w.\label{choixMxR}
\end{equation}

We claim that  {\bf wH(}$C${\bf )} holds with the shrinking parameters
$\io M_{x,R} $
and the constant $C = 2^{\gamma} \frac{C_s}{C_i  } $.

Fix $x_0\in\Gamma$, $R>0$, $M>M_{x_0, R}$ and
 let  $u$ be a positive harmonic function on $B(x_0, MR)$.

 First note that under {\bf (\ref{GEgamma})}, the graph is transient
and we can apply Boukricha's lemma 
(\cite{Bou79},  \cite[p. 37]{Tel06})
with $B_0=B(x_0,R), \ B_1=B(x_0,(M+1)R), \ 
B_2=B(x_0, (M+2)R)$ and $B_3=\Gamma$. So we get that if $u$ is harmonic
on $B_2$, then
$$   \max_{B_0} u \leq D \min_{B_0} u,$$ with
$$ D= \max_{x,y \in B_0} \max_{z\in \overline{B_2}\setminus B_1} \frac{G(x,z)}{G(y,z)}.$$
So, we have to compare $G(x,z)$ and $G(y,z)$ for $x,y\in B_0$
and $z \in B(x_0, (M+2) R)\setminus B(x_0, (M+1)R).$ 
For all $w\in B_0$, $D(w,z) > MR > R_w$ by (\ref{choixMxR}).
 Hence, by  {\bf (\ref{GEgamma})},
$$   \frac{C_i}{D(w,z)^{\gamma}}    \leq G(w,z) \leq    \frac{C_s}{D(w,z)^{\gamma}}. $$
Then, we successively  have :
\begin{eqnarray*}
G(x,z) &\leq &  \frac{C_s}{D(x,z)^{\gamma}}   \\
& = &  \frac{C_s}{C_i  }  \   \left(\frac{D(y,z)}{D(x,z)}\right)^{\gamma} 
\   \frac{C_i}{D(y,z)^{\gamma}}       \\ \\
 &\leq&     \frac{C_s}{C_i  }  \ \left( \frac{R+(M+2) R}{(M+1)R-R}\right)^{\gamma}\ G(y,z) \\
 &\leq&   \frac{C_s}{C_i  } \left(\frac{M+3}{M}\right)^{\gamma}\ G(y,z)\\
 &\leq & 2^{\gamma} \frac{C_s}{C_i  }  G(y,z).
 \end{eqnarray*}
\end{proof}

We can  now state  the main lemma to prove
Theorem I. 


\begin{lem} \label{lemma} 
Let $(\Gamma, a)$ be a weighted graph
which verifies {\bf wH(}$C${\bf)}. Fix $x_0\in \Gamma$.

Let $A$ be a finite subset of $\Gamma$.  Let $r_A>0 $ be such that $A\subset B(x_0,r_A)$.

For all $M> 2$, there is $\lambda_M >1$ such that for all $\lambda>
\lambda_M $ and for all  $ y\in A$ and $z\in \partial B(x_0,\lambda M r_A)$,
 \begin{equation}\label{maineq}
  P_y ( X_{\tau_A \wedge \tau_{\partial B}}=z | \tau_A >\tau_{\partial
    B}) 
= H_{\partial B}(x_0,z) [1+O \Big(M^{-\nu}\Big)],
\end{equation}
where $B=B(x_0,\lambda M r_A)$ and
$\nu>0$ is the Hölder exponent given by lemma \ref{holder}.
The constant in $O(\cdot)$ depends only on the constants $C$ and $c$
that appear  in   {\bf wH(}$C${\bf)} and in lemma \ref{holder} respectively.
\end{lem}

\begin{proof}  
For $M> 2$,
choose $M_2$ and $M_3$ such that
 $$M_2 > M(x_0 , M r_A)\  \hbox{ and } \     M_3 > M(x_0 , M_2 M r_A)$$
where  $M(x_0, \cdot )$ are the shrinking parameters that appear in {\bf wH(}$C${\bf )}.

Let $B_1 = B(x_0, M  r_A)$, $B_2 = B(x_0, M_2 M r_A)$ and
$B_3 = B(x_0, M_3 M_2 M  r_A)$.

For $z\in \partial B_3 $,  we consider the function 
$$f(x)= P_x ( X_{\tau_{\partial B_3}} =z), \quad x\in \Gamma. $$

Since $f$ is harmonic on $B_2$, by lemma \ref{holder}, 
   for all $u\in B_1$,
\begin{equation*} 
 |f(u)-f(x_0)|\leq c \Big( \frac{D(x_0, u)}{M r_A}\Big)^{\nu}
 \max_{B_2} f.  
\end{equation*}
In particular, for $u\in\partial B(x_0,  r_A)$,
\begin{equation} \label{1}
 |f(u)-f(x_0)|\leq  \frac{c}{M ^{\nu}}
 \max_{B_2} f.  
\end{equation}

Now by considering $f$ harmonic on $B_3$, 
since the graph verifies {\bf wH(C)}, we have that
\begin{equation} \label{2} 
\max_{B_2 } f \leq C f(x_0). 
\end{equation}

Therefore, by (\ref{1}) and (\ref{2}),  for all $u\in \partial
B(x_0, r_A) $, 
\begin{equation} \label{3}  
P_u ( X_{\tau_{\partial B_3}} =z) =   H_{\partial B_3}(x_0,z) \left[1+O \Big(
M^{-\nu}\Big) \right]. 
\end{equation}

Introduce the following notation.
For  $U, V$ and $W$ subsets of $\Gamma$ with $U\subset V \subset W.$
We put
\begin{equation}\label{crossing}
\partial V [W,U]=  \{ x \in \partial V;\text{ there   exist   paths
  in  $\Gamma$  from  $ x$  to $\partial W$ and from $x$  to $U$ }  \}.
\end{equation}
On the set $\{ \tau_A<\tau_{\partial B_3}\}$, we let $\eta=\inf\{ j\geq
\tau_A;\ X_j\in  \partial B(x_0, r_A) \}$. 

Then using (\ref{3}),
we obtain that for all $x\in \partial B(x_0,  r_A) [B_3 , A]$
\begin{eqnarray}\label{4}
P_x ( X_{\tau_{\partial B_3}} =z | \tau_A<\tau_{\partial B}) &=& \sum_{u\in \partial B(x_0,  r_A) } 
P_x ( X_{\eta} =u | \tau_A < \tau_{\partial B_3}) P_u ( X_{\tau_{\partial B_3}} =z )\nonumber
\\
&=&   H_{\partial B_3}(x_0,z) [1+O \Big( M^{-\nu}\Big)]  
\end{eqnarray}
Let $x \in \partial B(x_0, r_A) [ B_3, A ]$.
By  (\ref{3}) and (\ref{4}), we get
from the relation 
\begin{eqnarray*}
  P_x ( X_{\tau_{\partial B_3}} =z)  & =  & 
P_x ( X_{\tau_{\partial B_3}} =z|\tau_A > \tau_{\partial B_3}) P_x( \tau_A
> \tau_{\partial B_3} ) \\
 & & \qquad + P_x ( X_{\tau_{\partial B_3}} =z|\tau_A \leq \tau_{\partial B_3})(1-P_x(
\tau_A > \tau_{\partial B_3} )),
\end{eqnarray*}
that
 \begin{equation*} P_x ( X_{\tau_{\partial B_3}} =z|\tau_A > \tau_{\partial B_3}) =   H_{\partial B_3}(x_0,z) [1+O \Big( M^{-\nu}\Big)] 
\end{equation*}
This can also be written as,
\begin{equation}\label{4b}
P_x ( X_{\tau_{\partial B_3}\wedge \tau_A} =z  ) = 
  H_{\partial B_3}(x_0,z)P_x( \tau_A > \tau_{\partial B_3} ) 
  [1+O \Big( M^{-\nu}\Big)].
\end{equation}

Note that every path from $y$ to $\partial B_3$ must go through
some vertex
of $\partial B(x_0, r_A) [ B_3, A ]$.
So, for all $y \in A $ and for all $z\in \partial B_3$,
\begin{eqnarray*}
P_y ( X_{\tau_{\partial B_3}\wedge \tau_A} =z )
&=& \sum_{x \in \partial B(x_0, r_A) [ B_3, A ]} 
 P_y (X_{\tau_{\partial B(x_0, r_A) [ B_3, A ]}\wedge \tau_A} =x )P_x ( X_{\tau_{\partial B_3}\wedge \tau_A} =z )
\nonumber\\
&\stackrel {(\ref{4b})} =&   H_{\partial B_3}(x_0,z) [1+O \Big(M^{-\nu}\Big)] \nonumber \\
& \ & \hspace{0.5cm} \times \sum_{x \in \partial B(x_0,  r_A) [ B_3, A ]} 
 P_y (X_{\tau_{\partial B(x_0,  r_A) [ B, A ]}\wedge \tau_A} =x )P_x( \tau_A > \tau_{\partial B_3} ) \nonumber \\
 &=&  H_{\partial B_3}(x_0,z) [1+O \Big( M^{-\nu}\Big)]
 P_y( \tau_A > \tau_{\partial B_3} ).
\end{eqnarray*}
This last equation proves that lemma \ref{lemma} holds with
 $\lambda_M =  M_2 M_3$ where $M_2= M(x_0 , M r_A) $ and $M_3 = M(x_0 , M_2 M r_A)$.
\end{proof}

As in Lawler \cite[p. 49]{Law96}, using a last exit
decomposition, we obtain the following representation of the hitting
distribution in a weighted graph.

 Let $(\Gamma,a)$ be a   weighted graph. 
The Green function of the random walk in $B\subset \Gamma$
is defined by
$$
G_B(x,y)  := \sum_{k=0}^\infty p_B(x,y,k),\quad x,y\in \overline B
$$
where $ p_B(x,y,k) := P_x(X_k=y, k < \overline\tau_{B^c})$ are the transition probabilities
of the walk with Dirichlet boundary conditions.

Let $A\subset B$ be finite subsets of $\Gamma$. Then for all $x\in B^c$ and $y\in A$,
\begin{eqnarray}\label{laweq1}
H_A(x,y)=\sum_{z\in \partial B} G_{A^c}(x,z) H_{A\cup\partial B}(z,y),
\end{eqnarray}

\begin{eqnarray*}\label{laweq2}
{\overline H}_A(x,y)=\frac{ \sum_{z\in \partial B} G_{A^c}(x,z) H_{A\cup\partial B}(z,y)
}{ \sum_{z\in \partial B} G_{A^c}(x,z)P_z(\tau_A<\tau_{\partial B}) }
\end{eqnarray*}
and
\begin{eqnarray*}\label{laweq3}
\min_{z\in \partial B} \frac{ H_{A\cup\partial B}(z,y)}{ P_z(\tau_A<\tau_{\partial B}) }
\leq {\overline H}_A(x,y) \leq 
\max_{z\in \partial B} \frac{ H_{A\cup\partial B}(z,y)}{ P_z(\tau_A<\tau_{\partial B}) }.
\end{eqnarray*}

 Then by reversibility, $\io \pi(z)H_{A\cup\partial B}(z,y)=
\pi(y)H_{A\cup\partial B}(y,z)$
and\sm $\io P_z(\tau_A<\tau_{\partial B}) =\sum_{\tilde y \in
  A}H_{A\cup\partial B}(z,\tilde y)$.
Hence,
\begin{eqnarray}\label{habg4} 
\min_{z\in \partial B} \frac{ \pi(y)H_{A\cup\partial B}(y,z)}{
  \sum_{\tilde y\in A}\pi(\tilde y) H_{A\cup\partial B}(\tilde y,z) }
\leq {\overline H}_A(x,y) \leq 
\max_{z\in \partial B} \frac{ \pi(y)H_{A\cup\partial B}(y,z)}{
  \sum_{\tilde y\in A}\pi(\tilde y) H_{A\cup\partial B}(\tilde y,z) }
\end{eqnarray}

We complete the proof of Theorem I with the help of
(\ref{habg4}).

\begin{proof}[{\bf{Proof of Theorem I.}} ]
Let $A$ be a finite subset of $\Gamma$ and let $x_0\in \Gamma$.

Let $r_A>0 $ be such that $A\subset B(x_0, r_A)$.

Let $B=B(x_0, \lambda M r_A)$ where $\lambda\geq \lambda_M$ is given by lemma \ref{lemma}.

By equation (\ref{maineq}),   for all $y\in A$ and $z\in \partial B$,
\begin{equation} \label{eq2}
\pi(y) H_{A\cup \partial B} (y,z) =
   H_{\partial B}(x_0,z) [1+O \Big( M^{-\nu}\Big)]  \pi(y)   P_y( \tau_A > \tau_{\partial B} ).
\end{equation}

By summing over $y\in A$ the equation  (\ref{eq2}) gives,
\begin{equation} \label{eq1}
\sum_{y\in A } \pi(y) P_y ( X_{\tau_{\partial B}\wedge \tau_A} =z )=
   H_{\partial B}(x_0,z) [1+O \Big(M^{-\nu}\Big)]  \sum_{y\in A}  \pi(y) P_y( \tau_A > \tau_{\partial B} ).
\end{equation}

Since $(\Gamma, a)$ is connected, both sides of  (\ref{eq1}) are positive.  So we can divide   (\ref{eq2}) by (\ref{eq1}). And  a short calculation shows that
\begin{equation*}
\frac{ \pi(y) H_{A\cup \partial B} (y,z) }{\sum_{\tilde{y}\in A } \pi(\tilde{y})P_{\tilde{y}} ( X_{\tau_{\partial B}\wedge \tau_A} =z )} = 
\frac{\pi(y) P_y( \tau_A > \tau_{\partial B} ) }{\sum_{\tilde{y}\in A}  \pi(\tilde{y}) P_{\tilde y}( \tau_A > \tau_{\partial B} ) } [1+O \Big(N^{-\nu}\Big)]
\end{equation*}
where the constant in $O(\cdot)$ still depends only on
the constants $C$ and $c$
that appear  in   {\bf wH(}$C${\bf)} and in lemma \ref{holder} respectively. 

By (\ref{habg4}), we have that for all $v\notin B$,
\begin{eqnarray*}
\min_{z\in \partial B} \frac{ \pi(y)H_{A\cup\partial B}(y,z)}{ \sum_{\tilde{y}\in A}\pi(\tilde{y}) P_{\tilde{y}}( X_{\tau_{\partial B}\wedge \tau_A} =z ) }
\leq {\overline{H}}_A(v,y) \leq 
\max_{z\in \partial B} \frac{ \pi(y)H_{A\cup\partial B}(y,z)}{ \sum_{\tilde{y}\in A}\pi(\tilde{y}) P_{\tilde{y}} ( X_{\tau_{\partial B}\wedge \tau_A} =z ) }
\end{eqnarray*}

So for all $v\notin B$ we get:
\begin{equation}\label{eq4}  \overline{H}_A(v,y) =\frac{\pi(y) P_y( \tau_A
    > \tau_{\partial B} ) }{\mathop{\rm Cap}_B(A) }
 [1+O \Big(M^{-\nu}\Big)]
 \end{equation}
As $v$ goes to $+\infty$ in an arbitrary way, we will have that $M \to
\infty$ as well. Hence, by (\ref{eq4}),   we obtain that
$ \lim_{v\rightarrow +\infty}  \overline{H}_A(v,y)$ exists and
  $$ \lim_{v\rightarrow +\infty}  \overline{H}_A(v,y)= \lim_{m\rightarrow +\infty} \overline{H}_A^m(v,y)= 
  \frac{\pi(y) P_y( \tau_A > +\infty  ) }{\sum_{\tilde{y}\in A}
    \pi(\tilde{y}) P_{\tilde y}( \tau_A > +\infty ) }.$$
\vskip-5mm
\end{proof}


\section{Recurrent graphs}\label{recurrent}

In this section, we prove the existence of the harmonic measure
for  the random walk on a supercritical percolation cluster
of $\Z^2$. The proof for the uniformly elliptic random
walk on $\Z^2$ is similar but with many simplifications
since we can use the estimates of \cite{Del99}
instead of Barlow's estimates.

\subsection{Estimates of the capacity of a box}\label{Kestensgrid}

\begin{prop}\label{capalog}
Let $p>p_c({\mathbb Z}^2)$. 
There is a constant $C\geq 1$ such that 
${\mathbb P}_p $-a.s. for $x_0\in {\mathcal C}_\infty$,
for all $n$ sufficiently large,
\begin{equation}
  \label{capabds}
C^{-1} \leq (\ln n) {\mathop{\rm Cap}}_{B_\o(x_0,n)} ( \{x_0\} ) 
\leq
C.
\end{equation}
\end{prop}
  
Flows of finite energy on the supercritical percolation
cluster with respect to convex gauge functions
are constructed in \cite{ABBP06}.
To do so, the flow is expressed  by a probability
 on the set of self-avoiding paths.
Here, however, the lower estimate  of (\ref{capabds}) is obtained 
by combining the method used in $\Z^2$, see \cite[Proposition 2.14]{LyPe},
with a percolation lemma of Kesten \cite[Theorem 7.11]{Kes82}.

\begin{proof}
The upper bound follows from the variational principle and a
comparison with ${\mathbb Z}^2$ (see for instance \cite[section 3.1]{Tel06}).

To prove the lower bound, we assume $0\in\C_\infty$ and for each $n$ sufficiently large,
we construct a particular flow $\theta_n$ from $0$ to $\partial  B_\o(0,n)$.
However, it is a difficult task to estimate the energy of a flow from $0$ to $
\partial B_\o(0,n)$ consisting of small flows along simple paths from
$0$ to $ \partial B_\o(0,n)$ since  the percolation cluster   is very irregular.
 So, as in Mathieu and Remy \cite{MaRe04}, using a Theorem of Kesten
 \cite{Kes82}, we construct a grid of open paths in  $[-n;n]^2$  which
 we will call Kesten's grid.


{ \bf{Construction of Kesten's grid}} 

 Let us introduce some definitions.
 \begin{defn}  Let $B_{m,n}=([0;m]\times [0;n])\cap \Z^2$.

A horizontal [resp. vertical] channel of $B_{m,n}$ is a path in $\Z^2$
$(x_0,x_1,x_2,...,x_L)$  
such that:
\begin{enumerate}[$\bullet$]
 \item $\{x_1,x_2,...,x_{L-1}\} $ is contained in the interior of $B_{m,n}$
 \item $x_0\in \{0\}\times [0;n]$ [resp. $x_0\in [0;m]\times \{0\}$]
 \item $x_L\in \{m\}\times [0;n]$ [resp. $x_L\in [0;m]\times \{n\}$]
 \end{enumerate} 
 \end{defn}
 We say that two channels are disjoint if they  have no vertex in
 common. Let $N(m,n)$ be the maximal number of disjoint open
 horizontal channels in $B_{m,n}$. 
Then by \cite[Theorem 11.1] {Kes82}, for $p>p_c$,
there is a constant $c(p)$ and some universal constants
 $0<c_4, c_5, \xi < \infty$, such that

 \begin{equation}
\P_p\big(N(m,n)>c(p)n \big)\geq 1-c_4 (m+1) \exp(-c_5(p-p_c)^{\xi}n).\label{kesten}
\end{equation}

We apply this result to the number of disjoint open channels
in a horizontal strip of length $n$ and width $C_K\ln n$ contained in
 $[-n;n]^2$. If $C_K$ is large enough so that
$c_5(p-p_c)^\xi C_K >3$ then
$$ \sum_n n  c_4 (n+1)\exp(-c_5(p-p_c)^{\xi}C_K
\ln n)<\infty.$$

Hence by (\ref{kesten}) and  Borel-Cantelli lemma, we get that for $n$ large
 enough there is at least $c(p)C_K \ln n$ disjoint open channels in each
 horizontal strip of  length $n$ and width $C_K\ln n$ contained in
 $[-n;n]^2$. We do the same construction for vertical strips. 
Finally, we obtain  that each horizontal and each vertical strip of length $n$
 and width $C_K\ln n$ in $[-n;n]^2$  contains at least  $c(p) C_K\ln n$ disjoint open
 channels.

Fix $n$ large enough so that there is a Kesten's grid
and  let $J$ be the largest integer such that
$$(J+1/2) C_K\ln n < n.$$

Set $N := (J+1/2) C_K\ln n$.
Divide  $\io [- N ;   N]^2$ into squares $S_{i,j}$ of side $C_K \ln n $
 centered at $(iC_K\ln n; j C_K\ln n )$ for $-J\leq i, j \leq J$, $i,j\in\Z$.
Denote   this set of  $(2J+1)^2$  squares by ${\cS}_n$.

Since $B_\o(0,N)  \subset \C_{\infty} \cap [-N ;  N]^2$,
we have that $\io\capa_{\C_{\infty} \cap [-N ;  N]^2 }(0)\leq \capa_{B_\o(0,N)}(0)$.
Hence, to obtain a lower bound, 
it suffices to construct a flow $\theta_n$ from $0$ to
the vertices of $\C_{\infty}$ on the sides of $ [-N ;  N]^2$. \\

{ \bf{Construction of the flow}} 

To each open path $\Pi : (x_0,x_1,x_2, ... ,x_L)$
from $x_0=0$ to a side of $ [-N ;  N]^2$ with the induced orientation,
we  associate the unit flow  with source at $0$,
$\Psi^{\Pi}= \sum_{\ell=1}^L (\1_{\{\er_\ell\}} -\1_{\{\el_\ell\}})$ where $\er_\ell$ is the edge
from $x_{\ell-1}$ to $x_{\ell}$.
The flow $\theta_n$  will be a sum of  flows 
$\Psi^{\Pi}$ for a set $\mathcal{P}_n$ of well chosen open paths.

 \begin{defn}
A sequence   $(S_k ; -m\leq k\leq m)$  of squares of $\cS_n$ is called a  {\it{path of squares}}
 if for $-m\leq k < m$,  the squares  $S_k $ and $S_{k+1}$ have a common side.
 \end{defn}

We now proceed in three steps.
First, to each square of $\cS_n$ on the left side of $[-N ; N]^2$,
 we construct a path of squares of $\cS_n$  from left to right containing the square
 centered at $(0;0)$.
See figure 1.
 In the second step, for each of these paths of squares, we contruct $c(p)C_K\ln n$
open paths using Kesten's grid. And in the last step, we show how to
connect
these open paths to $(0;0)$.
 
 \begin{center}
 \begin{figure} [h!]\label{chemboit}
\includegraphics[width=6cm,height=9cm]{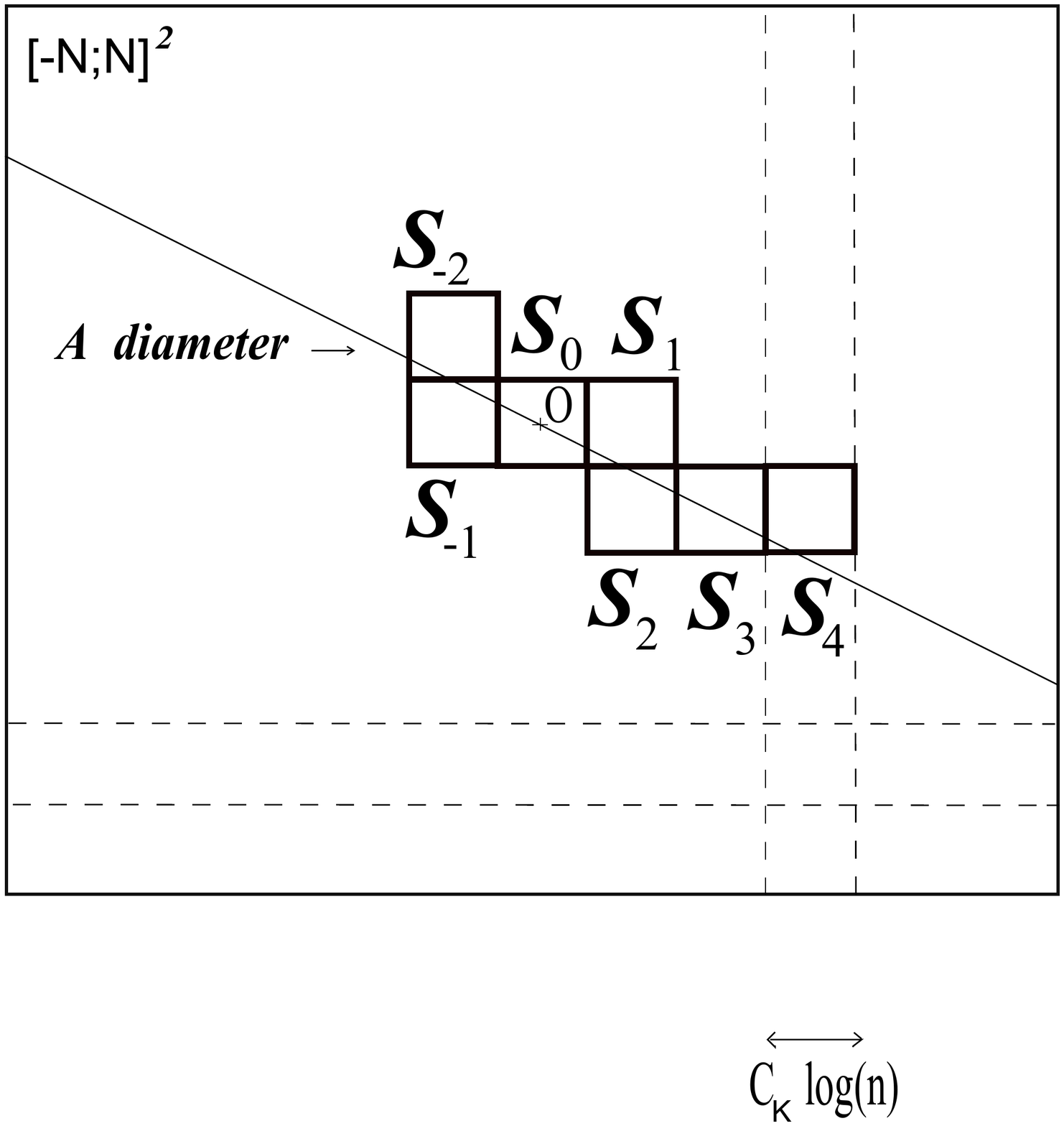}
\vskip-2cm\caption{A diameter and some squares of the associated path of squares.}
\end{figure}
\end{center}

We call  {\it{diameter}}, a line segment  from  $( -N  ; j C_K\ln n )$
 to $(N; -j C_K\ln n; )$ for $-J\leq j \leq J$.
 
 To each diameter, we associate a path of squares
$(S_k; -m\leq k\leq m)$
consisting of the squares of $\cS_n$ which intersect the diameter.
Whenever a diameter goes through a corner of
 a square of $\cS_n$,  add an additional square (there are two choices)
 in order to obtain a path of squares.

Now fix one of these $2J+1$ paths of squares. 

Let $S_k, S_{k+1},
\ldots, S_{k+K-1} $, with $2\leq K\leq 2 J+1$, be a horizontal stretch of
squares within a path of squares that is, a horizontal path of squares
of maximal length $K$. Then each of the $c(p)C_K\ln n$ horizontal open channels
crossing this horizontal strip contains an open path from the leftmost side
of $S_k$ to the rightmost side of $S_{k+K-1}$ which lies entirely inside
$S_k \cup S_{k+1}\cup
\ldots \cup S_{k+K-1} $. Indeed, running along an open channel from left to
right,
it consists of the edges of the channel between the last visited vertex of the leftmost side
of $S_k$ to the first visited vertex of the rightmost side
of $S_{k+K-1}$. See case (i) of   figure 2. 
We proceed similarly for all horizontal and vertical stretches of the
path of squares.

Whenever the path of squares turns, the vertical and
horizontal open paths of the corresponding stretches must be
connected together. For instance,
to go right along a horizontal stretch to down along a vertical stretch,
the lowest horizontal channel is attached to the leftmost vertical
channel, the second lowest horizontal channel is attached
to the second leftmost vertical
channel and so on.  See case (ii) of   figure 2. 
We proceed similarly for the other turns. 
The open paths obtained by this procedure  might not be disjoint but
each  of their edges belongs to at most two paths.
Actually, a slightly more complicated rule
would yield disjoint open paths.
\begin{center}
 \begin{figure} [h!]\label{two}
\includegraphics[width=7cm,height=6.5cm]{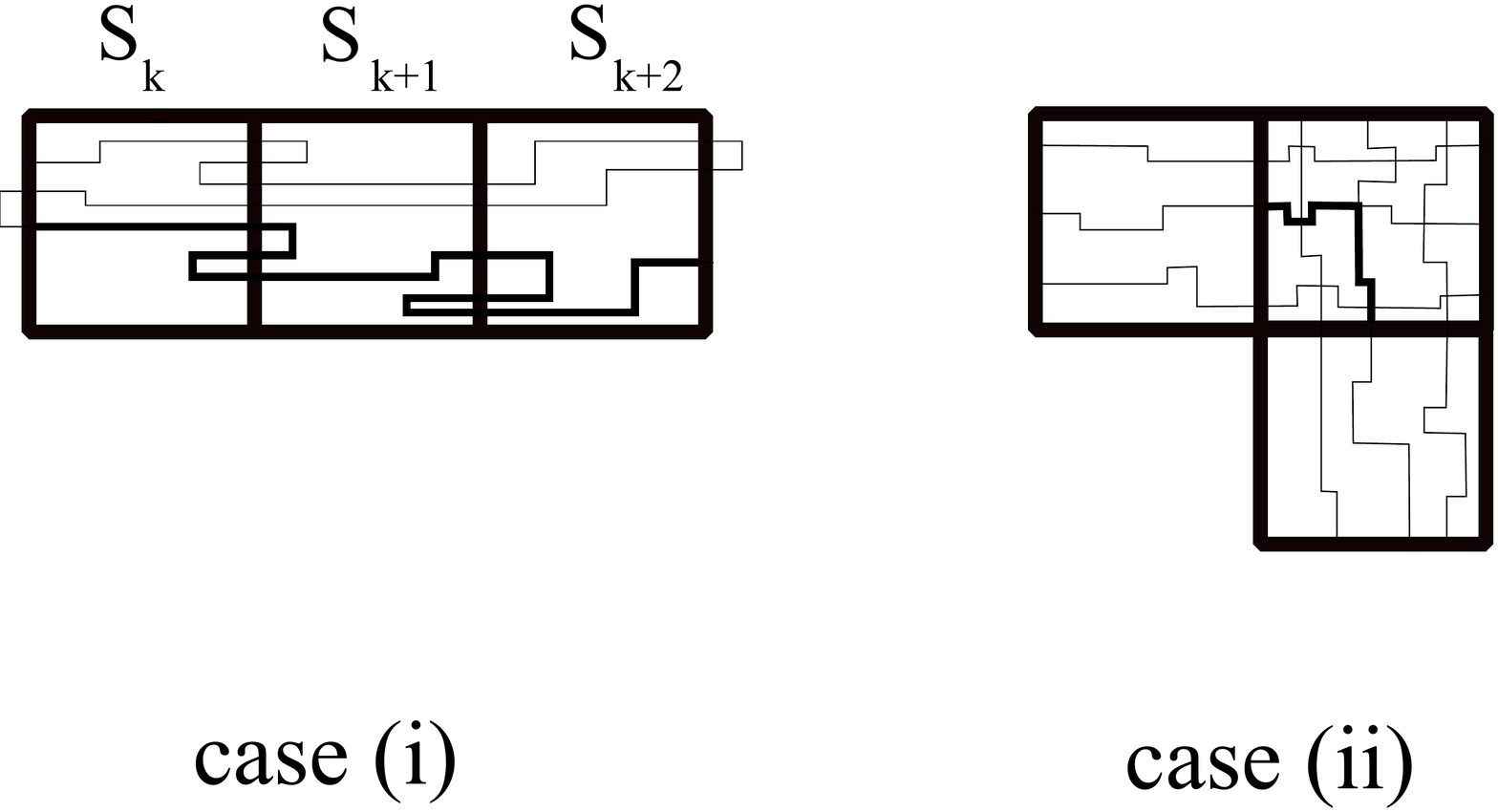}
\vskip-0.6cm\caption{ In bold, the chosen portion of the open channels.}
\end{figure}
\end{center}
The second step is completed and we have constructed a set $\mathcal{P}'_n$ of
$(2J+1) c(p)C_K\ln n$ open
paths from  the left  to the right  side of $[-N ; N]^2$.

The last step is to connect these paths to $(0;0)$.
Choose two vertices of the infinite cluster $\C_\infty$, $x_1$ and $x_2$.
$x_1$ is chosen in the square of $\cS_n$ centered at $(0;C_K\ln
n)$ and $x_2$  in the square of $\cS_n$ centered at $(0;-C_K\ln n)$.
Then for $ \ell  = 1, 2$, let $\pi_\ell$ be an open path from
$x_\ell$  to $(0;0)$ for which the chemical distance is attained.
By  (\ref{AP2}), for $n$ large enough,
\begin{equation}
  \label{longpi}
  \hbox{$\pi_1$ and $\pi_2$ have lengths less than}\ \ 4\mu C_K \ln n.
\end{equation}

Since each path of $\mathcal{P}'_n$ intersects $\pi_1\cup\pi_2$,
we can associate to each path of $\mathcal{P}'_n$ two open paths from
$(0;0)$, one to the left side of $[-N ; N]^2$ and the other one to the right side. 
A specific rule would be to follow a path of $\mathcal{P}'_n$ starting
on the left side of $[-N ; N]^2$ up to its first intersection with
$\pi_1\cup\pi_2$
which we then follow up to $(0;0)$ and for the second path, start
on the right side of $[-N ; N]^2$ up to the first intersection with
$\pi_1\cup\pi_2$ which we then follow up to $(0;0)$.

The set $\mathcal{P}_n$  consists of all these open paths from
$(0;0)$
to the left or the right side  of $[-N ; N]^2$ and we set
$$ \theta_n = \sum_{\Pi \in \mathcal{P}_n } \Psi^{\Pi}.$$ 
The intensity of the flow $\theta_n$ is
\begin{equation}\label{intense}
(2J+1)c(p)C_K \ln n\geq Cn.
\end{equation}

{ \bf{The energy of the flow}} 

Let us now obtain an upper bound on the energy of the flow $\theta_n$. 

A square $S_{i,j}\in\cS_n$, $-J\leq i,j\leq J$,  belongs to less than
$C\  J /(\vert  i\vert +1)$ paths of squares for some constant $C>0$
independent of $n$. Hence for $-J\leq i,j\leq J$,
\begin{equation}\label{aa} 
\hbox{each edge of}\  S_{i,j}\ \hbox{ belongs to less than} \ \ 
C\  J /(\vert  i\vert +1)\  \hbox{paths of}\   \mathcal{P}'_n.
\end{equation}

Finally, gathering (\ref{longpi}),  (\ref{intense}), (\ref{aa}), by Thomson's principle (see for instance \cite[section
2.4]{LyPe}), we get 

\begin{eqnarray*}
\dfrac{1}{ {\mathop{\rm Cap}}_{B_\o(0,N)} ( 0)} 
&\leq&  \frac{C}{n^2} \sum_{e \in \lbrack -N ; N\rbrack^2}  \theta_n(e)^{2} \\
&\leq &  \frac{C}{n^2} \left[ \sum_{-J\leq i,j\leq J}  \sum_{ e\in S_{i,j}}
\theta_n(e)^2+   \sum_{ e\in \pi_1\cup\pi_2 } \theta_n(e)^2\right] \\
&\leq &  \frac{C}{n^2} \left[\  \sum_{1\leq j\leq J} j \left(\frac{J}{j}\right)^2  
  ( 2C_K \ln n)^2 +  n^2   \ln n  \ \right] \\
&\leq &  C\ln n
\end{eqnarray*}
\end{proof}

\subsection{The Green kernel and its properties}

In this section, we will use the parabolic Harnack inequality for the
random walk on the supercritical 
percolation cluster proved by Barlow and Hambly in \cite{BaHa09}.

Besides this, we also use the comparison result for $D$
and the $\vert\cdot\vert_1$-distance
of Antal and Pisztora \cite{AnPi96}, see (\ref{AP2}).

\begin{lem}\label{recgreen} $\P_p$-almost surely, for all
$x_0, x\in \C_{\infty}$, the series 
\begin{equation}
  \label{grseries}
  \sum_{k=0}^\infty \left[  p(x_0,x_0,k) - p(x,x_0,k)\right]
\end{equation}
converges. The limit will be denoted by $g(x,x_0)$.

Let $G_{2n}(x,y)$ and $p_{2n}(x,y,k)$ be respectively 
the Green function and the probability transitions of the random walk in the ball
$B(x_0, 2n)$
with Dirichlet boundary conditions. Then
\begin{eqnarray}\label{greenker}
  g(x,x_0)   =   \lim_n \sum_{k=0}^\infty \left[  p_{2n}(x_0,x_0,k) - p_{2n}(x,x_0,k)\right]
     =  \lim_n  [G_{2n}(x_0,x_0) - G_{2n}(x,x_0)] 
\end{eqnarray}
\end{lem}

\begin{proof}
Let $R_0$ be given by  the Harnack inequality for  the supercritical
cluster (\ref{HarnackBarlow2}). Then as in the proof of
\cite[Proposition 6.1]{Bar04}, we have that for $x\in \C_{\infty}$ and $R\geq
R_0(x)$,
$B(x,R)$ is very good (see \cite [definition 1.7]{Bar04}) 
with $N_B\leq R^{1/(10(d+2))}$ and it is exceedingly good  (see \cite [definition 5.4]{Bar04}) .

Now let $R\geq R_0(x) \vee 16$ and let $R_1= R\ln R$.
Then, since $R_1\geq R_0$,  $B=B(x, R_1)$ is very good with 
$N_B^{2d+4}\leq R_1^{(2d+4)/(10(d+2))}\leq R_1/(2\ln R_1)$.
Then by \cite[Theorem 3.1]{BaHa09}, there exists a constant $C_H$
such that the parabolic Harnack inequality \cite[(3.2)]{BaHa09}  holds
in $Q(x,R,R^2)$.
Therefore \cite[Proposition 3.2]{BaHa09} holds with $s(x_0) = R_0(x_0) \vee 16$
and $\rho(x_0,x) = R_0(x_0) \vee 16 \vee D(x_0, x)$

Fix $x_0\in \C_\infty$ then $v(n,x) := p(x,x_0,n) + p(x,x_0, n+1)$
 is a caloric function, that is, it verifies
$$v(n+1,x) - v(n,x) = \L v(n,x),\quad (n,x) \in {\mathbb N}\times
\C_\infty.$$

Let $k>4 D(x_0, x)^2$. Let $t_0= k+1$ and $r_0 = \sqrt {t_0}$.
Then $v(n,x)$ is caloric in $\rbrack 0, r_0^2\rbrack\times
B(x_0,r_0)$,
$x\in B(x_0 , r_0/2)$ since $D(x_0, x) \leq \sqrt{k} < r_0 / 2$,
and $t_0 - \rho(x_0, x)^2\leq k\leq t_0 - 1 $.

Then by the upper gaussian estimates \cite[Theorem 5.7]{Bar04} and
\cite[(2.18)]{BaHa09} and by \cite[Proposition 3.2]{BaHa09}, 
there is $\nu>0$ such that
   \begin{eqnarray*}
     \vert v(k,x) - v(k, x_0)\vert 
  & \leq & C \left( \frac {\rho(x_0,x)}{\sqrt{t_0}}\right)^\nu\sup_{Q_+}
  v\\
  & \leq & C \left( \frac
    {\rho(x_0,x)}{\sqrt{t_0}}\right)^\nu\frac{1}{r_0^2}\\
& \leq & C\   \frac {\rho(x_0,x)^\nu}{k^{1+\nu/2}}.
   \end{eqnarray*}
Note that we also have that for all $k>4 D(x_0, x)^2$,
 $$\vert p(x,x_0,k)-p(x_0,x_0,k)\vert \leq  C\   \frac
 {\rho(x_0,x)^\nu}{k^{1+\nu/2}}.$$ 
Hence (\ref{grseries})   converges.
Then (\ref{greenker})   follows by Lebesgue dominated convergence theorem.
\end{proof}

\begin{lem}
There is a constant $c_6\geq 1$ such that, $\P_p$-a.s.,
for all $x_0\in \C_{\infty} $ there is $\rho=\rho(x_0)$ such that if $D(x_0,x)>\rho$,
\begin{equation}
  \label{asympg2}
  c_6^{-1} \ln D(x_0,x)  \leq  g(x,x_0) \leq  c_6 \ln D(x_0,x) 
\end{equation}
\end{lem}

\begin{proof}
Let $x_0\in \C_{\infty} $ and $r>0$.
Write  $\io\tau_{r}:=\inf \{ k> 0 ; X_k\in B(x_0,r)\}$
and for $m\geq 1$, write 
  $\io\sigma_{m}:=\inf \{ k> 0 ; X_k\notin B(x_0,m)\}$.

Note that for all $n>3\mu m$ where $\mu$ is the constant  
 that appears in (\ref{AP1}), $P_{\cdot}(\sigma_{n}< \tau_{r})$ is harmonic in
$\± B(x_0, 3\mu m)\setminus B(x_0,r) $.
Then by the annulus Harnack inequality (proposition
\ref{annulusharnack1}),
if $m$ is sufficiently large and $D(y, x_0)=m$, then
$\io  \sum_{x\in B(x_0,r)}   \pi(x) P_{x}(\sigma_{n}< \tau_{r}) $
\begin{eqnarray*}
   & = & \sum_{x\in  B(x_0,r)}    \sum_{x'; D(x_0, x') =m} 
\pi(x) P_{x}(X(\sigma_m)=x', \sigma_m< \tau_{r}) P_{x'}(\sigma_{n} < \tau_{r})\\
& \asymp & P_y(\sigma_{n} < \tau_{r}) \sum_{x\in  B(x_0,r)}
\sum_{x'; D(x_0, x') =m}   \pi(x) P_{x}(X(\sigma_m)=x', \sigma_m< \tau_{r})\\
 & \asymp &  P_y(\sigma_{n} <  \tau_{r})\capa_m(B(x_0,r)).
\end{eqnarray*}

By $f_1(y,m,n)\asymp f_2(y,m,n)$ here, we  mean that there is a
constant $c\geq 1$ 
which does not depend on $y,m,n$ nor on $\omega$ and $r$,
 and such that $\P_p$-a.s for $m$ is sufficiently large and $D(y, x_0)=m$, then
 $$0< c^{-1} f_1(y,m,n) \leq f_2(y,m,n)\leq c f_1(y,m,n).$$

Then by  the capacity estimates (\ref{capabds}), for $m = D(y, x_0) $,
\begin{eqnarray}
  \label{tobeimpr}
  P_y(\sigma_{n}<\tau_{r}) \asymp \frac{\capa_{n}(B(x_0,r))}{
  \capa_m(B(x_0,r))}\asymp\frac {\ln m}{ \ln n} = \frac {\ln D(y, x_0)}{ \ln n}.
\end{eqnarray}

It follows from (\ref{tobeimpr})  and the capacity estimates
(\ref{capabds}) that  for $m$ sufficiently large, $D(x,x_0)=m$
and   $n>3\mu m$, 
$$ \begin{array}{rcl}
G_{n}(x_0,x_0) - G_{n}(x,x_0) & = & G_{n}(x_0,x_0) -P_x(\tau_{x_0} < \sigma_{n}) G_{n}(x_0,x_0)\\
  & = & G_{n}(x_0,x_0) P_x(\tau_{x_0} > \sigma_{n})\\
   & \asymp & \ln n \frac{\ln D(x, x_0)}{\ln n}\\
\end{array}
$$

Then (\ref{asympg2}) follows by (\ref{greenker}).
\end{proof}

We will need to work in sets defined in terms of $g(\cdot,x_0)$.
  Let 
  \begin{equation}\label{eqtilB}
    \widetilde B_n :=\widetilde B(x_0, n)   :=   \{x\in \C_{\infty} ;
  g(x,x_0) < \ln n\} \ \hbox{and} \   \widetilde \sigma_n  :=  \inf\{k\geq 0 ;
  X_k\notin \widetilde B(x_0, n)\}.
  \end{equation}
Note that by (\ref{asympg2}), for all $n$ sufficiently large,
\begin{eqnarray}\label{ballincl}
  B(x_0, n^{1/c_6})\subset \widetilde B(x_0, n) \subset B(x_0, n^{c_6}).
\end{eqnarray}

\begin{lem}
There is a constant $C > 0$ such that, $\P_p$-a.s.,
for any  non empty finite subset $A$  of $\C_{\infty} $ and $x_0\in
 A$,  if $m$ sufficiently large and $n>(3\mu m)^{c_6}$, then
\begin{eqnarray*}\label{borinf}
  \min_{y; m = D(y, x_0) }   P_y(\widetilde \sigma_n< \tau_A)  \geq   C(\ln m / \ln n ).
\end{eqnarray*}
\end{lem}
\begin{proof}
The lemma is a consequence of   (\ref{tobeimpr}). For a finite subset
$A$ of $\C_{\infty} $ such that $x_0\in A\subset B(x_0,r)$, for $m$
sufficiently large and $n>(3\mu m)^{c_6}$,
\begin{eqnarray*}
  \min_{y; m = D(y, x_0) }   P_y(\widetilde \sigma_n< \tau_A) 
 & \geq  &  \min_{y; m = D(y, x_0) }  P_y(\widetilde \sigma_n<  \tau_{B(x_0,r)})\\
 & \geq  &   \min_{y; m = D(y, x_0) }   P_y( \sigma_{B(x_0, n^{1/c_6})}<  \tau_{B(x_0,r)})\\
 & \geq &  C(\ln m / \ln n^{1/c_6} )
\end{eqnarray*}
\vskip-5mm
\end{proof}

The harmonic measure will be expressed in terms of the function $u_A$
defined below.

\begin{defn}\label{defuA}
$\P_p$-a.s.,
for a finite subset $A$   of $\C_{\infty}(\o)$ and for a fixed $x_0\in
A$, let
  $$u_A (x,x_0)  := g(x,x_0) - E_{x, \o}
  g(X_{\overline \tau_A},x_0),\quad x\in\C_{\infty}(\o).$$
\end{defn}

Note that,
\begin{eqnarray*}
   & & u_A(\cdot , x_0) = 0 \quad\hbox{on}\quad A,\\
 & & u_A(x, x_0) \asymp \ln D_\o(x_0 , x) \quad\hbox{as}\quad D_\o(x_0 , x)\to
 \infty \quad\hbox{by}\ \ (\ref{asympg2}),\\
 & & Pu_A(x,x_0)= Pg(x,x_0) - \sum_{y\sim x} p(x,y) E_y g(X_{\overline\tau_A},x_0) 
 \\
 &   & \qquad\qquad =g(x,x_0) -\1_{x_0}(x)- E_x g(X_{\tau_A},x_0), \quad
x\in\C_{\infty}(\o).
\end{eqnarray*}

The next lemma is the analogue of \cite[Proposition 6.4.7]{LaLi10}.

\begin{prop}\label{limlnuA}
$\P_p$-a.s., for a finite subset $A$   of $\C_{\infty}(\o)$ and for $x_0\in A$
 and  $x\in A^c$,
$$ u_A(x,x_0 ) = \lim_n\  (\ln n) P_x(\widetilde\sigma_n<\tau_A).
$$
\end{prop}

\begin{proof}
Let $R_0(z,\o)$ be as the  Harnack inequality for  the supercritical
cluster  (\ref{HarnackBarlow2}).
By   (\ref{AP1}) of Antal and
Pisztora, and by (\ref{ballincl})
$$\sum_n \sum_{z\in\partial \widetilde B(x_0, n)}\P_p(z\in\C_{\infty}, R_0(z, \cdot ) \geq n^{1/c_6})\leq C\sum_n n^{2c_6}\exp(-c_3 n^{\eps/c_6}) <\infty.
$$

Therefore, by Borel-Cantelli, 
there is $\Omega_1 \subset\Omega$ with $\P_p(\Omega_1)=1$, 
 such that for all $\omega\in\Omega_1$ there is $n_0$ such that for
 all $n\geq n_0$ and for all $z \in\partial \widetilde B(x_0,n)$,
 $R_0(z)< n^{1/c_6}$.

Let $z\in\partial \widetilde B(x_0, n)$ where $n\geq n_0$. Then there is 
$z'\in \widetilde B(x_0, n)$ such that $z'\sim z$ and
$$g(z',x_0) < \ln n \leq  g(z,x_0).$$
Moreover, by (\ref{ballincl}), $D(z, x_0)> n^{1/c_6}$.
Then by Hölder's continuity property given in proposition
\ref{holderperco} and by (\ref{asympg2}),
\begin{eqnarray}
0\leq \  g(z,x_0) - \ln n  & \leq  &  g(z,x_0) - g(z',x_0)\nonumber \\
 & \leq  &  c \  \left( \frac{1}{n^{1/c_6}}  \right)^{\nu}
 \max_{B(z,n^{1/c_6})} g(\cdot,x_0)\nonumber \\
 & \leq  &  \frac c  C\  \left( \frac{1}{n^{1/c_6}} \right)^{\nu}\ln n.\label{firsteq}
\end{eqnarray}

By the optional stopping theorem applied to the martingale $g(X_k,x_0)$,
$k\geq 0$ and for $n$ large enough and $ x\in \widetilde B(x_0, n)\setminus A$,
\begin{eqnarray}
  g(x,x_0) & = & E_x \left[g(X_{\overline \tau_A\wedge \widetilde \sigma_n},
    x_0)\right],\nonumber\\
 & = & P_x(\widetilde \sigma_n<\tau_A) E_x \left[g(X_{ \widetilde \sigma_n},
    x_0)\mid \widetilde \sigma_n<\tau_A \right]\nonumber \\
 & &  + P_x(\tau_A < \widetilde \sigma_n ) E_x \left[g(X_{\tau_A},
    x_0)\mid \tau_A <\widetilde \sigma_n  \right] .\label{secondeq}
\end{eqnarray}
But
\begin{eqnarray*}
 \lim_n P_x(\tau_A < \widetilde \sigma_n )   E_x \left[g(X_{ \tau_A},
    x_0)\mid \tau_A <\widetilde \sigma_n  \right] 
  & = &  \lim_n  E_x \left[g(X_{ \tau_A},
    x_0) ; \tau_A <\widetilde \sigma_n  \right] \\
& = &  E_x g(X_{\tau_A},
    x_0) 
\end{eqnarray*}
Therefore by (\ref{firsteq}) and (\ref{secondeq}), $ u_A(x,x_0 ) = \lim_n (\ln n) P_x(\widetilde\sigma_n<\tau_A).$
\end{proof}

We can  now prove  the analogue of  lemma \ref{lemma} for the
supercritical cluster. Theorem  II  will follow
from this lemma and from proposition \ref{limlnuA} above.


\begin{lem} \label{lemma2d} 
Let $p>p_c(\Z^2)$. 
Let $\Omega_1$ and $R_0(x,\omega)$ be as in  the Harnack
inequality for the percolation cluster (\ref{HarnackBarlow2}).
 There is $\nu'>0$ such that the following holds.

Let $\o\in\Omega_1$ and let $A$ be a finite subset of  $\C_\infty(\o)$.
Fix $x_0\in A$.

Then there is $N_0=N_0(x_0, A, \o)$ such that for all $n>N_0$,
for all  $ y\in A$ and $z\in \partial \widetilde B(x_0, n)$,
 \begin{equation}\label{maineq2d}
  H_{A\cup \partial \widetilde B(x_0,n)}(y, z) 
= P_y( \tau_A > \widetilde\sigma_{n} )  H_{\partial \widetilde B(x_0,n)}(x_0,z) \left[1+   O \Big( \frac{\ln n}{n^{\nu'}} \Big) \right]
\end{equation}
where $\widetilde B_n$ and $\widetilde \sigma_n$ are as in (\ref{eqtilB}).
$\nu'>0$ depends on  the Hölder exponent given by proposition
\ref{holderperco}
and the constants given in (\ref{asympg2})
The constant in $O(\cdot)$ depends on $\o$ and $A$ and on the constants 
that appear  in (\ref{HarnackBarlow1}),  (\ref{HarnackBarlow2}) 
and in proposition \ref{holderperco}.
\end{lem}

\begin{proof}  
Let $m$ be sufficiently large so that $A\subset B(x_0, m)$ 
and so that (\ref{borinf}) holds for all  $n>(3\mu m)^{c_6}$.

 For  $R_1> \max\{ R_0(x_0,\omega), (3\mu m)/4,  m \}$,
let $B_1 = B(x_0,R_1)$, $B_2 = B(x_0, 2R_1)$,
$B_3 = B(x_0, 4R_1) $.
Set $n={(4R_1)^{c_6}}$ and let  $\widetilde B_n = \widetilde B(x_0, n)$ and
$\widetilde \sigma_n$ be as in (\ref{eqtilB}). Note that
by (\ref{ballincl}),
$B_3\subset \widetilde B_n$ and (\ref{borinf}) holds.

For $z\in \partial \widetilde B_n $,  consider the function 
$$f(x)= P_x ( X_{\widetilde\sigma_{n}} =z), \quad x\in \C_\infty(\o). $$

Since $f$ is harmonic on $B_2$, by proposition \ref{holderperco}, 
   for all $u\in B_1$,
\begin{equation*} 
 |f(u)-f(x_0)|\leq c \Big( \frac{D(x_0, u)}{R_1}\Big)^{\nu}
 \max_{B_2} f.  
\end{equation*}
In particular, for $u\in\partial B(x_0,  m)$,
\begin{equation} \label{12d}
|f(u)-f(x_0)|\leq c \Big( \frac{m}{R_1}\Big)^{\nu}
 \max_{B_2} f. 
\end{equation}

Now  by considering $f$ harmonic on $B_3$, 
by (\ref{HarnackBarlow1}), we have that
\begin{equation} \label{22d} 
\max_{B_2 } f \leq c_1 f(x_0). 
\end{equation}

Therefore, by (\ref{12d}) and (\ref{22d}),  for all $u\in \partial
B(x_0, m) $, 
\begin{equation} \label{32d}  
P_u ( X_{\widetilde\sigma_{n}} =z) =   H_{\partial \widetilde B_n}(x_0,z) \left[1+O \Big(
\Big( \frac{m}{R_1}\Big)^{\nu}\Big) \right]. 
\end{equation}

On the set $\{ \tau_A<\widetilde\sigma_{n}\}$, we let $\eta=\inf\{ j\geq
\tau_A;\ X_j\in  \partial B(x_0, m) \}$. 

Then using (\ref{32d}),
we obtain that for all $x\in \partial B(x_0,  m) [\widetilde B_n , A]$ (see
(\ref{crossing}) for the notation),
\begin{eqnarray}\label{42d}
P_x ( X_{\widetilde\sigma_{n}} =z | \tau_A<\widetilde\sigma_{n}) &=& \sum_{u\in \partial B(x_0,  m) } 
P_x ( X_{\eta} =u | \tau_A < \widetilde\sigma_{n}) P_u ( X_{\widetilde\sigma_{n}} =z )\nonumber
\\
&=&   H_{\partial \widetilde B_n}(x_0,z) \left[1+O \Big(
\Big( \frac{m}{R_1}\Big)^{\nu}\Big) \right].
\end{eqnarray}
Let $x \in \partial B(x_0, m) [ \widetilde B_n , A ]$.
By  (\ref{32d}), (\ref{42d}) and (\ref{borinf}), we get
from the relation 
\begin{eqnarray*}
  P_x ( X_{\widetilde\sigma_{n}} =z)  & =  & 
P_x ( X_{\widetilde\sigma_{n}} =z|\tau_A > \widetilde\sigma_{n}) P_x( \tau_A
> \widetilde\sigma_{n} ) \\
 & & \qquad + P_x ( X_{\widetilde\sigma_{n}} =z|\tau_A \leq \widetilde\sigma_{n})(1-P_x(
\tau_A > \widetilde\sigma_{n} )),
\end{eqnarray*}
that
 \begin{eqnarray*} 
P_x ( X_{\widetilde\sigma_{n}} =z|\tau_A >\widetilde\sigma_{n}) 
 &  =  &   H_{\partial \widetilde B_n}(x_0,z) \left[1+\frac{1}{P_x(\tau_A >\widetilde\sigma_{n})}    O \Big(
\Big( \frac{m}{R_1}\Big)^{\nu}\Big) + O \Big(
\Big( \frac{m}{R_1}\Big)^{\nu}\Big)\right]\\
&  =  &   H_{\partial \widetilde B_n}(x_0,z) 
\left[1+   O \Big( \frac{\ln n}{\ln m} \Big(
  \frac{m}{R_1}\Big)^{\nu}\Big) \right]\\
&  =  &   H_{\partial \widetilde B_n}(x_0,z) 
\left[1+   O \Big( \frac{\ln n}{n^{\nu'}} \Big) \right]
\end{eqnarray*}
where $\nu'= \nu/c_6 >0$ and where the constant in the last $O(\cdot)$
now depends on 
$\o$ and $A$.
This can also be written as,
\begin{equation}\label{4b2d}
P_x ( X_{\widetilde\sigma_{n}\wedge \tau_A} =z  ) = 
  H_{\partial \widetilde B_n}(x_0,z)P_x( \tau_A > \widetilde\sigma_{n} ) 
 \left[1+   O \Big( \frac{\ln n}{n^{\nu'}} \Big) \right].
\end{equation}

Note that every path from $y\in A$ to $\partial \widetilde B_n$ must go through
some vertex
of $\partial B(x_0, m) [ \widetilde B_n , A ]$.
So, for all $y \in A $ and for all $z\in \partial \widetilde B_n$,
\begin{eqnarray*}
P_y ( X_{\widetilde\sigma_{n}\wedge \tau_A} =z )
&=& \sum_{x \in \partial B(x_0, m) [ \widetilde B_n, A ]} 
 P_y (X_{\tau_{\partial B(x_0, m) [ \widetilde B_n , A ]}\wedge \tau_A} =x )P_x ( X_{\widetilde\sigma_{n}\wedge \tau_A} =z )
\nonumber\\
&\stackrel {(\ref{4b2d})} =&   H_{\partial \widetilde B_n}(x_0,z) \left[1+   O \Big( \frac{\ln n}{n^{\nu'}} \Big) \right]\nonumber \\
& \ & \hspace{0.5cm} \times \sum_{x \in \partial B(x_0,  m) [
  \widetilde B_n , A ]} 
 P_y (X_{\tau_{\partial B(x_0,  m) [ \widetilde B_n, A ]}\wedge \tau_A} =x )P_x(
 \tau_A > \widetilde\sigma_{n} )
 \nonumber \\
 &=&  H_{\partial \widetilde B_n}(x_0,z) \left[1+   O \Big( \frac{\ln n}{n^{\nu'}} \Big) \right]
 P_y( \tau_A > \widetilde\sigma_{n} ).
\end{eqnarray*}
Hence (\ref{maineq2d}) holds with $N_0=(4 \max\{ R_0(x_0,\omega), (3\mu m)/4,  m \})^{c_6}$.
\end{proof}

\subsection{The existence of the harmonic measure}

We now show how to obtain  Theorem  II  from lemma \ref{lemma2d}.

\begin{proof}
Let $y\in A$.
Let $\widetilde B_n$ and $\widetilde \sigma_n$ be as in
(\ref{eqtilB}).
 For $x \notin \widetilde B_n$, by (\ref{laweq1}), by reversibility of the
 Markov chain and  by (\ref{maineq2d}), for all $n > N_0$,
\begin{eqnarray}
 \pi(x)  H_A(x,y)&  =  &  \pi(x)  P_x(X_{\tau_A} = y )\nonumber \\
 & =  &  \pi(x)  \sum_{z\in \partial \widetilde B_n}
  G_{A^c}(x,z)  H_{A\cup
  \partial \widetilde B_n}(z,y) \nonumber\\
& = &  \sum_{z\in \partial \widetilde B_n}
 G_{A^c}(z,x) \pi(y) H_{A\cup
  \partial \widetilde B_n}(y,z) \nonumber\\
& = &  \sum_{z\in \partial \widetilde B_n}
 G_{A^c}(z,x)\pi(y) P_y(\widetilde\sigma_n< \tau_A)   H_{\partial \widetilde B_n}(x_0,z)  \left[1+O \Big(
n^{-\nu'}\Big) \right]\nonumber\\
  & = &  \pi(y) P_y(\widetilde\sigma_n< \tau_A)  \sum_{z\in \partial \widetilde B_n}
 G_{A^c}(z,x) H_{\partial \widetilde B_n}(x_0,z)  \left[1+O \Big(
n^{-\nu'}\Big) \right].\label{sansx}
\end{eqnarray}

At this point for the supercritical cluster of ${\mathbb Z}^d$, $d\geq
3$,
it suffices to sum over $y\in A$ and divide the equations.
However, since the walk is recurrent on the supercritical
percolation cluster of $\Z^2$, $P_{y}(\widetilde\sigma_n< \tau_A)\to 0$
as $n\to \infty$, this would lead to an indeterminate limit.
But by (\ref{sansx}),
\begin{eqnarray*}
  H_A(x,y)  & = & \frac {\pi(x)  H_A(x,y)} {\pi(x) \sum_{y'\in A} H_A(x,y') } \\
  & = &   \frac{\pi(y) P_y(\widetilde\sigma_n< \tau_A)
  }{\sum_{y'\in A}\pi(y') P_{y'}(\widetilde\sigma_n< \tau_A) }  \left[1+O \Big(
n^{-\nu'}\Big) \right]
\end{eqnarray*}
and by proposition \ref{limlnuA},
\begin{eqnarray}
  \lim_{x\to\infty}   H_A(x,y)  
 & = & \lim_{n\to\infty}\frac{(\ln n) \pi(y) P_y(\widetilde\sigma_n< \tau_A)
  }{(\ln n) \sum_{y'\in A}\pi(y') P_{y'}(\widetilde\sigma_n< \tau_A) }\left[1+O \Big(
n^{-\nu'}\Big) \right] \nonumber\\
 & = & \frac{ \pi(y) Pu_A(y,x_0) }{ \sum_{y'\in A}\pi(y') P u_A(y',x_0) }.\label{ident}
\end{eqnarray}
\end{proof}

\section{Proof of proposition \ref{annulusharnack1}}\label{proofAH1}
\bi
In this proof, we keep the notations of \cite{Bar04} except for the
graph
distance which will still be denoted by $D(x,y)$.

For  a cube $Q$ of side $n$, let
$Q^+ := A_1\cap {\mathbb Z}^d$ and  $Q^\oplus := A_2\cap {\mathbb
  Z}^d$ 
where $A_1$ and $A_2$ are the cubes in ${\mathbb R}^d$ with the same
center as $Q$ and with side length $\frac{3}{2} n$ and $\frac {6}{5}n$
respectively.
Note that $Q\subset Q^\oplus\subset Q^+$.

$\C(x)$ is the connected open cluster that contains $x$.
${\mathcal C}_Q(x)$, which will be called the open $Q$ cluster,
 is the set of vertices connected to $x$ by an open path
within $Q$.
And ${\mathcal C}^\vee(Q)$ is the largest open  $Q$ cluster
(with some rule for breaking ties).

Set $\alpha_2= (11(d+2))^{-1}$.

\begin{proof}
By \cite[lemma 2.24]{Bar04} and by Borel-Cantelli lemma,
for all $x\in\Z^d$, there is $N_x$ such that for all $n>N_x$,
$L(Q)$ (see \cite[p. 3052]{Bar04}) holds for all cubes $Q$ of side $n$
with $x\in Q$.

Let $z\in\Z^d$ and let $n> N_z=N_z(\o)$.

Let $Q$ be a cube of side $n$ which contains $z$.

Let $x_0\in {\mathcal C}^\vee(Q^+)\cap Q^{\oplus}$
with $Q(x_0, r+k_0)^+\subset Q^+$ where $C_H n^{\alpha_2} \leq r \leq
n$ and $k_0=k_0(p,d)$ is the integer chosen in \cite[p. 3041]{Bar04}.

Let $R$ be such that
\begin{equation}
  \label{conditEH1}
  B_\omega(x_0, (3/2)R\ln R) \subset Q^{\oplus}\qquad\hbox{and}
\end{equation}
\begin{equation}
  \label{conditEH2}
  (C_H n^{\alpha_2})^{d+2} \leq (C_H n^{\alpha_2})^{4(d+2)} < R <
  R\ln R < n.
\end{equation}
Then by \cite[Theorem 2.18c]{Bar04}, 
$B_\omega(x_0,R\ln R)$ is $(C_V, C_P, C_W)$- very good
with $$N_{B_\omega (x_0, R\ln R)} \leq C_H n^{\alpha_2}$$
with the constants given in \cite[section 2]{Bar04}.

Then by \cite[Theorem 5.11]{Bar04} and (\ref{conditEH2}), 
there is a constant $C_1$, which depends only on $d$ and on the
constants $C_V, C_P, C_W$, such that
 if $D(x_0,x_1) \leq \frac { 1 }{3} R\ln R$
and if $h:\overline{ B(x_1, R)} \to {\mathbb R}$ is positive and harmonic in
$B(x_1, R)$, then
\begin{equation}
\max_{B(x_1, R/2)} h \leq C_1 \min_{B(x_1, R/2)} h.\label{HarCv}
\end{equation}
Note that since $4\alpha_2(d+2) = 4/11 < 1/2$, the conditions (\ref{conditEH2}) are
verified for $R=2\sqrt n$ when $n$ large enough.
 
We now apply a standard chaining argument to a well chosen covering by
balls (see for instance
\cite[chapters 3 and 9]{Tel06}). Let $x_0\in\Z^2$ and consider
environments
such that $x_0\in \C_\infty(\o)$.
The main difficulty to carry out the chaining argument is to check
that  the intersection of ``consecutive''  balls is not empty. 
The remainder of the proof is to construct an appropriate
covering of $ \{x\in\C_\infty  ; D(x_0,x)=m\}$, for $m$ large enough, 
 with a finite number balls,
which does not depend on $x_0$, $m$ or $\o$,  and such that
 the Harnack inequality (\ref{HarCv}) holds in each ball.

Let $\delta_1, \delta_2$ and $\delta_3$ be three positive real numbers
such that
\begin{equation}\label{deltas}
  2\delta_2<\delta_1\quad\hbox{and}\quad\delta_1+ 2\delta_2  < \delta_3 < \frac{1}{5\mu}\left(\frac{4}{5} - \delta_2\right).
\end{equation}
For instance, choose $\delta_3$ so that $0<\delta_3< 4/(50\mu)$, then
choose
$\delta_1$ so that $0<2\delta_1<\delta_3$ and finally 
choose $\delta_2 $ so that $\delta_2< \min\{\delta_1/2 , 4/(50\mu)\}$.

Let $n>N_{x_0}$.

Furthermore, take $n$ large enough so that there is a Kesten's grid in $Q$
with constant $C_K$ and $R(Q)$ holds (by \cite[lemma 2.8]{Bar04}). 
That is in each vertical and each horizontal
strip of width $C_K\ln n$ contains at least $c(p)C_K\ln n$ open disjoint
channels. Moreover, since $R(Q)$ holds,
${\mathcal C}^\vee(Q)\subset {\mathcal C}^\vee(Q^+)$.
In particular, $x_0\in {\mathcal C}^\vee(Q^+)\cap Q^{\oplus}$.

Furthermore by (\ref{AP1}) and Borel-Cantelli, if $m$ is large enough
then for all $x,y\in \C_\infty$ such that $ \vert x\vert_1\leq 3\mu m$,
$\vert y\vert_1\leq 3\mu m$ and $\io \vert x-y\vert_1\geq m(\delta_1-2\delta_2)/\mu$
we have $$\vert x-y\vert_1 \leq D(x,y) \leq \mu \vert x-y\vert_1.$$
Set $\io \frac{R}{2}= m\delta_3 = \sqrt n$.

Furthermore, take $m$ large enough so that 
(\ref{conditEH1}) and (\ref{conditEH2}) are verified as well as
$$C_K\ln n < m\delta_2 /\mu,\quad
3m\mu<\frac{1}{3} R\ln R\quad\hbox{and} \quad r < 4m\delta_3.$$

Instead of constructing a finite covering of $\{x\in\C_\infty ; D(x_0,x) = m\}$, 
it is easier to construct a finite covering of
the region $\{x\in\C_\infty ; \frac {4m}{5\mu}   \leq \vert x - x_0\vert_1 \leq 2 m\}$
which is a larger subset of $\Z^2$.

Let $\I := \{ (i;j) \in {\mathbb N}^2 ; 4/(5\delta_1)   \leq  i+j \leq 2\mu/\delta_1.\}$ Let $M$ be the
cardinal of $\I$.

Let $x_{i,j} = x_0 + (i m\delta_1/\mu; j m\delta_1 / \mu)$ with $(i;j) \in \I$.
Then for each $x_{i,j}$ with  $ (i;j) \in \I$, there is $\widetilde x_{i,j}
\in \C_\infty$ such that $\vert x_{i,j} - \widetilde x_{i,j}\vert_1
\leq m\delta_2/\mu$.

We proceed similarly in the other three quadrants to obtain
 a set of $4M$ vertices which we denote by ${\mathcal D}$.
Note that $M$ does not depend on $m$.
 
The finite covering of the region $\frac {4m}{5\mu}   \leq \vert x -
x_0\vert_1 \leq 2 m$
is $$\{B(\widetilde x , m\delta_3 ),\quad \widetilde x \in {\mathcal
  D}\}.$$
Note that each ball contains 
the center of the four neighbouring balls
 except those on the boundary of the region. But these are connected
 to at least one neighbouring ball. Indeed, if $\widetilde x,
\widetilde y\in {\mathcal D}$ are neighbouring centers then by (\ref{deltas}),
$$D(\widetilde x, \widetilde y)<\mu\vert \widetilde x- \widetilde
y\vert
<m(\delta_1 + 2\delta_2) < m\delta_3 .
$$

If  $\widetilde x\in {\mathcal D}$ then  by (\ref{deltas}),
$$D(x_0,\widetilde x)>\frac{m}{\mu}\left(\frac{4}{5} -
  \delta_2\right)> 5m\delta_3,$$
 $D(x_0,\widetilde x)<\mu\vert x_0 - \widetilde x\vert_1< 2m\mu$
and $\mu\left(2m+m\delta_2/\mu\right)< 3m\mu$.

Therefore, $B(x_0,r)  $ does not belong to a ball of the covering and
$u$ is harmonic in each ball $B(\widetilde x , 2m\delta_3 )$
with $\widetilde x \in {\mathcal
  D}.$
Then the Harnack inequality holds for 
$\io R = 2m\delta_3 $
since for all $\widetilde x \in{\mathcal D}$,
$$D(x_0,\widetilde  x)< 2m\mu<
\frac{1}{3} R\ln R.$$
\end{proof}

\vskip1cm
{\it{Acknowledgment}}: The authors would like to thank Pierre Mathieu
for numerous discussions and particularly,
for pointing out the usefulness of Kesten's lemma.
\bino
 This research was supported by the French ANR projects MEMEMO and
  MEMEMO2.

\bibliographystyle{plain}
\bibliography{HD-bibliography}

\vskip2cm
\begin{center}
\begin{tabular}{llll}
Daniel Boivin &  &  &  Clément Rau\\
Université Européenne de Bretagne &  &  & Universit\'e Paul Sabatier \\
Université de Bretagne Occidentale &  & & Institut de Math\'ematiques de Toulouse\\
Laboratoire de Mathématiques CNRS UMR 6205 &  & & route de Narbonne\\
6 avenue Le Gorgeu, CS93837 &  &  & 31400 Toulouse\\
 F-29238 Brest Cedex 3,&  & &  France \\
 France &  & &   \\
 &  & &   \\
boivin@univ-brest.fr & &  &  rau@math.ups-tlse.fr\\
http://stockage.univ-brest.fr/~boivin/   & & & http://www.math.univ-toulouse.fr/~rau/\\
\end{tabular}
\end{center}
 
\end{document}